\documentclass[11pt]{amsart}
\usepackage{epsfig}
\usepackage{graphicx, subfigure}
\usepackage{amsmath, amssymb,mathabx,mathrsfs}
\usepackage{color}
\usepackage[toc,page]{appendix}
\usepackage{comment}
\usepackage{hyperref}

\newtheorem{thm}{Theorem}[section]

\newtheorem{lem}[thm]{Lemma}

\theoremstyle{definition}

\theoremstyle{remark}
\newtheorem{rem}[thm]{Remark}
\numberwithin{equation}{section}

\newcommand{\R}{\mathbb R}
\newcommand{\T}{\mathbb T}
\newcommand{\Z}{\mathbb Z}

\newcommand{\eps}{\varepsilon}

\newcommand{\X}{\mathcal{X}}

\newcommand{\st}{{\rm s}}           
\newcommand{\un}{{\rm u}}           

\newcommand{\tM}{ \widetilde{M}}

\newcommand{\tz}{ \tilde{z}}

\newcommand{\tPhi}{\tilde{\Phi}}
\newcommand{\tGamma}{\tilde{\Gamma}}
\newcommand{\tsigma}{\tilde{\sigma}}

\newcommand{\bF}{\mathbf{F}}

\def\tLambda{ {\tilde \Lambda}}

\begin{document}
\title[]{Global effect of non-conservative perturbations on 
homoclinic orbits}
\author{Marian\ Gidea$^\dag$}
\address{Yeshiva University, Department of Mathematical Sciences, New York, NY 10016, USA }
\email{Marian.Gidea@yu.edu}
\thanks{$^\dag$ Research of M.G. was partially supported by NSF grant  DMS-1814543.}
\author{Rafael de la Llave${^\ddag}$}
\address{School of Mathematics, Georgia Institute of Technology, Atlanta, GA 30332, USA}
\email{rafael.delallave@math.gatech.edu}
\thanks{$^\ddag$ Research of R.L. was partially supported by NSF grant DMS-1800241,
 and H2020-MCA-RISE \#734577}
 \author{Maxwell Musser$^\dag$}
\address{Yeshiva University, Department of Mathematical Sciences, New York, NY 10016, USA }
\email{mmusser@mail.yu.edu}
\thanks{$^\dag$ Research of M.M. was partially supported by NSF grant  DMS-1814543.}

\subjclass[2010]{Primary,
37J40;  
37C29; 34C37; 
Secondary,
70H08. 
} \keywords{Melnikov method; homoclinic orbits; scattering map.}
\date{}

\dedicatory{To the memory of Florin Diacu.}

\begin{abstract}
We study the effect of time-dependent, non-conservative perturbations on the dynamics along homoclinic orbits to a normally hyperbolic invariant manifold. We assume that the unperturbed system is Hamiltonian, and the normally hyperbolic invariant manifold is parametrized via action-angle coordinates.
The homoclinic excursions can be described via the scattering map, which gives the future asymptotic of an orbit as a function of the past asymptotic. We provide explicit formulas, in terms of convergent integrals, for the perturbed scattering map expressed in action-angle coordinates. We illustrate these formulas in the case of a perturbed rotator-pendulum system.
\end{abstract}

\maketitle
\section{Introduction}

\subsection{Brief description of the  main results and   methodology}
In this paper we  study the effect of small, non-conservative, time-dependent perturbations on the dynamics along  homoclinic  orbits in Hamiltonian systems.  
We describe this dynamics via the scattering map, and estimate the effect of the perturbation on the scattering map.
We illustrate the computation of the perturbed scattering map on a simple model: the  rotator-pendulum system.
However, similar computations can be obtained for more general systems. 

Our approach is based on geometric methods and on Melnikov theory. The geometric framework assumes the following situation. There exists a normally hyperbolic invariant manifold (NHIM) whose stable and unstable manifolds coincide. The orbits in the intersection are homoclinic orbits which are bi-asymptotic to the normally hyperbolic invariant manifold.
To each homoclinic   intersection we can associate a scattering map.
By definition, the scattering map assigns to the  foot-point  of the unstable fiber passing through a given homoclinic point the foot-point   of the stable fiber  passing through the same homoclinic point. The scattering map is a diffeomorphism of an open subset of the NHIM onto its image.
If the  system is Hamiltonian and the NHIM is  a symplectic manifold then the scattering map is a symplectic map. If a small, Hamiltonian perturbation is added to the system, the scattering map remains symplectic, provided that the NHIM persists under the perturbation.
This is no longer the case when a non-conservative perturbation is added to the system: the perturbed scattering map -- provided that it survives the perturbation -- does not need to be symplectic.

In the rotator-pendulum model that we consider, the NHIM can be\break parametrized via action-angle coordinates, so the scattering map can be described in terms of these coordinates as well. In the unperturbed case, the  scattering map is  the identity. Then we add  a small, non-conservative, time-dependent perturbation. Using Melnikov theory, we  estimate the effect of the perturbation on the scattering map  to first order with respect to the size of the perturbation. We provide  expressions for  the difference between the perturbed scattering map and the unperturbed one, relative to the action and angle coordinates,
in terms of convergent improper integrals of the perturbation evaluated along  homoclinic  orbits of the unperturbed system. One important aspect in the computation is that, in the perturbed system, the action is a slow variable, while the angle is a fast variable.

Similar computations of the scattering map, in the case when the perturbation is Hamiltonian, have been done in, e.g. \cite{DelshamsLS08a}. The effect of the perturbation on the action component of the scattering map is relatively easy to compute directly. On the other hand, the effect on the angle component of the scattering map is more complicated to compute, since this is a fast variable. To circumvent this difficulty, the paper \cite{DelshamsLS08a} uses the symplecticity of  the scattering map to estimate indirectly the effect of the perturbation on the angle component. In our case, since we consider non-conservative perturbations, this type of argument no longer holds. We therefore perform a direct computation of the effect  of the perturbation on the angle component of the scattering map.

\subsection{Related works}

The Melnikov method has been developed to study the  persistence of
periodic orbits and of homoclinic/heteroclinic connections under periodic perturbations \cite{Melnikov63}. 

One well-known application of the Melnikov method is to show that degenerate homoclinic orbits in the unperturbed system yield transverse homoclinic orbits in the perturbed system, see, e.g.,  \cite{HolmesM82,guckenheimer1984nonlinear,Robinson88,Wiggins90,
DelshamsRamirezRos1996,DelshamsRamirezRos1997,DelshamsG00,DelshamsG00a}.
The effect of the homoclinic orbits is given in terms of certain improper integrals
referred to as `Melnikov integrals'. In some of these papers the integrals are only conditionally
convergent, and the sequence of limits of integration must
be carefully chosen in order to obtain the correct dynamic meaning. 

Another important application of the Melnikov method is to estimate the effect of the perturbations on the scattering map, which is associated to homoclinic excursions to a normally hyperbolic invariant manifold. In the case when the perturbation is given by a time-periodic or quasi-periodic Hamiltonian, this effect is estimated in, e.g.,  
\cite{DelshamsLS00,DelshamsLS06a,DelshamsLS06b,DelshamsLS08a,DLS_Multi,GideaLlaveSeara14,delshams2017arnold}. 
The effect on the scattering map of general time-dependent Hamiltonian perturbations is studied in, e.g.,  \cite{GideaL17,gidea2018global}.

Some other  papers of a related interest include
\cite{BaldomaFontich1998,LomeliMR08,LomeliMeiss2000,Roy2006,LomeliMR08,granados2012melnikov,
granados2014scattering,granados2017invariant}.

A novelty of our paper is that we study the effect on the scattering map of general time-dependent perturbations that can be non-conservative. The methodology used in some of the earlier papers, which relies on the symplectic properties of the scattering map, does not extend to the non-conservative case.

We also note that the results here are global in the sense that they apply to all homoclinics to a NHIM, while  
other results only apply to homoclinics to  fixed points or periodic/quasi-periodic orbits.

\subsection{Structure of the paper}
In Section \ref{section:setup} we provide the set-up for the problem, and describe the model that makes the main focus of the subsequent results -- the rotator-pendulum system subject to general time-dependent, non-conservative perturbations. In Section~\ref{sec:preliminaries}, we describe the main tools -- normally hyperbolic invariant manifolds and the scattering map. In Section~\ref{sec:master_lemmas} we provide some lemmas that are used in the subsequent calculations. The main results are formulated and proved in Section~\ref{sec:scattering_perturbed}. Theorem \ref{thm:transverse_homoclinic} gives sufficient conditions for the existence of transverse homoclinic intersections for the perturbed system. Theorem \ref{prop:change_in_I} provides estimates on the effect of the perturbation on the action-component of the scattering map. Theorem~\ref{prop:change_in_theta} provides estimates on the effect of the perturbation on the angle-component of the scattering map.
In Section~\ref{sec:comparison_similar} we show that, when the perturbation is Hamiltonian, the  formulas obtained in Theorem~\ref{prop:change_in_I} and  Theorem \ref{prop:change_in_theta} are equivalent to the corresponding formulas   in \cite{DelshamsLS08a}.

\section{Set-up}\label{section:setup}
Consider a $C^{r+1}$-smooth manifold  $M$ of dimension $(2m)$, where  $r\geq r_0$ for some suitable $r_0$.
Each point $z\in M$ is described via a system of local coordinates $(u,v)\in\R^{2m}$, i.e., $z=z(u,v)$.
Assume  that $M$ is  endowed with  the standard symplectic form
\begin{equation}\label{eqn:standard_symplectic}
\Upsilon=du\wedge dv=\sum_{i=1}^{m}du_i\wedge dv_i,\end{equation}
defined on local coordinate charts.

On $M$ we consider a non-autonomous system of differential equations
\begin{equation} \label{eqn:generalperturbation}
\dot z = \X_\eps(z;\eps)=\X^0(z) + \eps\X^1(z,t;\eps),
\end{equation}
where $\X^0:M\to TM$ is a $C^r$-differentiable vector field on $M$, $\X^1:M\times \R\times \R\to TM$ is a time-dependent, parameter dependent $C^r$-differentiable vector field on $M$,
and $\epsilon\in\R$ is a `smallness' parameter,  taking values in some  interval $(-\eps_0,\eps_0)$ around $0$.
Moreover, we assume that $\X^1=\X^1(z,t;\eps)$ is
uniformly differentiable in all variables.

The flow of \eqref{eqn:generalperturbation} will be denoted by $\Phi^t_\eps$.

Above, the  dependence of $\X^1(z,t;\eps)$ on the time $t$ is assumed to be of a general type, not necessarily periodic or quasi-periodic.
In the particular case of a periodic perturbation, we require that $t$ is defined mod 1, or, equivalently
$t\in \T^1$. In the particular case of a  quasi-periodic perturbation,  we require that the vector field $\X^1$ is
of the form $\X^1(z,\chi(t);\eps)$, for\break $\chi:\R\to\T^k$ of the form $\chi(t)=\phi_0+t\varpi$ for some $k\geq 2$, $\phi_0\in \T^k$ and
$\varpi\in \R^k$ a rationally independent vector, i.e., satisfying the following condition:   $h\in\Z^k$ and $h\cdot \varpi=0$ imply $h=0$.

Below, we will consider some situations when the  vector fields $\X^0$, $\X^1$ satisfy additional assumptions.

\subsection{The unperturbed system}\label{section:unperturbed} We assume that the vector field $\X^0$ represents an autonomous Hamiltonian vector field, that is, $\X^0=J\nabla_z H_0$ for some $C^{r+1}$-smooth Hamiltonian function $H_0:M\to \R$, where $J$ is an almost complex structure compatible with the standard symplectic form given by \eqref{eqn:standard_symplectic}, and the gradient $\nabla$ is with respect to the associated Riemannian metric\footnote{$g(u,v)=\omega(u,Jv)$.}.

Below we describe some of the geometric structures that are the subject of our study. These geometric structures are defined in Section \ref{sec:NHIM}.

\begin{itemize}
\item[\textbf{(H0-i)}] There exists a $(2d)$-dimensional manifold $\Lambda_0\simeq D\times\T^d\subseteq M$  that is a normally hyperbolic invariant manifold (NHIM) for the Hamiltonian flow $\Phi^t_0$ of $H_0$,
where $D$ is a closed $d$-dimensional ball $B^d$ in $\mathbb{R}^d$. 
\item[\textbf{(H0-ii)}] The manifold $\Lambda_0$ is parametrized
via action-angle coordinates, and is foliated by $d$-dimensional invariant tori, each torus corresponding to a fixed value of the action. The flow $\Phi^t_0$ on each such torus is a linear flow.

\item [\textbf{(H0-iii)}] The unstable and stable manifolds $W^\un(\Lambda_0)$, $W^\st(\Lambda_0)$    of $\Lambda_0$ coincide, i.e.,  $W^\un(\Lambda_0)=W^\st(\Lambda_0)$, and moreover, for each $z\in\Lambda_0$, $W^\st(z)=W^\un(z)$.
\end{itemize}

Condition \textbf{(H0-i)} says that there exists a NHIM for the flow. Condition \textbf{(H0-iii)} says that there exist  homoclinic orbits to the NHIM which are degenerate, as they   correspond  to the unstable and stable manifolds of the NHIM which coincide. We will show that if  the perturbation $\X^1$ satisfies some verifiable conditions, then  the unstable and stable manifolds of the perturbed NHIM intersect transversally for all $\eps\neq 0$ sufficiently small, so there exist transverse homoclinic orbits to the NHIM. The goal will be to quantify the effect of the perturbation  on the dynamics along  homoclinic orbits. This effect will be measured in terms of the changes in the action and angle coordinates when the orbit follows a homoclinic excursion. 
 
As a model for a system with the above properties, we consider the rotator-pendulum system, which is described
in detail in Section \ref{section:rotator_pendulum}.

\subsection{The perturbation.} The vector field  $\X^1$  is a time-dependent, parameter-dependent vector field on $M$. In the general case we will not assume that $\X^1$ is Hamiltonian, so the system \eqref{eqn:generalperturbation} can be subject to dissipation or forcing.

We will also derive results for the particular case when the perturbation $\X^1$  in \eqref{eqn:generalperturbation} is Hamiltonian, that is, it is given by
\begin{equation}\label{eqn:h}
   \X^1(z,t;\eps)=J\nabla_z H_1(z,t;\eps),
\end{equation}
where $H_1$ is a time-dependent, parameter-dependent $C^{r+1}$-smooth Hamiltonian function on $M$.

\subsection{Model: The rotator-pendulum system.}\label{section:rotator_pendulum}
This model is described by an autonomous Hamiltonian $H_0$ of the form:
\begin{equation}\label{eqn:rotator_pendulum}
\begin{split}
H_0(p,q,I,\theta,t)&=h_0(I)+h_1(p,q)\\&=h_0(I)+\sum_{i=1} ^{n}\pm\left (\frac{1}{2}p_i^2+V_i(q_i)\right),
\end{split}
\end{equation}
with $I=(I_1,\ldots,I_d)\in\mathbb{R}^d$, $\theta=(\theta_1,\ldots,\theta_d)\in \mathbb{T}^d$,
$p=(p_1,\ldots, p_n)\in \mathbb{R}^{n}$, $q=(q_1, \ldots, q_n)\in \mathbb{T}^{n}$, and $z=z(p,q,I,\theta)$.
In the above the sign $\pm$  means that for each $i$ there is some fixed choice of a sign $\pm$ in front of $\left (\frac{1}{2}p_i^2+V_i(q_i)\right)$.

The phase space $\R^{2(d+n)}$ is endowed with the symplectic form \[\Upsilon=\sum_{i=1}^{n}dp_i\wedge dq_i+\sum_{j=1}^{d}dI_j\wedge d\theta_j.\]


In the above, we assume the following:
\begin{itemize}
\item[(V-i)]
Each potential $V_i$ is periodic of period $1$ in~$q_i$;
\item[(V-ii)] Each potential $V_i$ has a non-degenerate local  maximum (in the sense of Morse),  which, without loss of generality,
we set at $0$; that is,   $V'_i(0)=0$ and  $V''_i(0)<0$. The non-degeneracy in the sense of Morse means  that, additionally,  $0$ is the only critical point in the level set $\{V_i(q) = V_i (0)\}$, that is,   $V'_i(q_*) = 0$ and $V_i(q_*) = V_i(0)$ implies $q_* = 0$.
\end{itemize}
Condition (V-ii) implies that each
pendulum   has a homoclinic orbit to $(0, 0)$.

We note that for the classical rotator, the standard assumption is that  $\partial^2 h_0/\partial I^2$ is positive definite; in our case  we allow that  $\partial^2 h_0/\partial I^2$ is of indefinite sign. For this reason we refer to $h_0$ as a `generalized' rotator. This situation appears in several applications, such as critical inclination of satellite orbits, quasigeostrophic flows, plasma devices, and transport in magnetized plasma   \cite{kyner1968rigorous,del1992hamiltonian,hazeltine2003plasma}.

For the classical pendulum, the Hamiltonian is of the form $\left (\frac{1}{2}p_i^2+V_i(q_i)\right)$; in our case we allow a sign $\pm 1$  in front each pendulum, so $\partial ^2h_1/\partial p^2$ can  be of indefinite sign. This is why we refer to the  terms in $h_1$ as  `generalized penduli'.

In Section \ref{section:NHIM_ unperturbed rotator_pendulum} we will show that for each closed $d$-dimensional ball $D\subseteq \R^d$, the set
\begin{equation}
\label{eqn:pendulm_NHIM_0}
\Lambda_0=\{(p,q,I,\theta)\,|\, I\in D,\, p=q=0\}
\end{equation}
is a NHIM with boundary.
The stable and unstable manifolds coincide, i.e.,  $W^\un(\Lambda_0)=W^\st(\Lambda_0)$,
and, moreover, for each $z\in\Lambda_0$,  $W^\un(z)=W^\st(z)$.
Each point in $W^\un(\Lambda_0)=W^\st(\Lambda_0)$ determines a  homoclinic trajectory which approaches $\Lambda_0$ in both positive and negative time.

We note that the geometric structures described above satisfy the properties
\textbf{(H0-i)}, \textbf{(H0-ii)}, \textbf{(H0-iii)}
in Section \ref{section:unperturbed}.

\section{Preliminaries}\label{sec:preliminaries}

\subsection{Vector fields as differential operators}\label{section:notation}
In the sequel,  we will identify vector fields  with
differential operators, which is a standard operation in differential geometry (see, e.g., \cite{BurnsG05}).
That is,  given a smooth vector field $\X$ and a smooth function $f$ on the manifold $M$,
\begin{equation}
\label{eqn:vf_derivative}(\X  f)(z) = \sum_j (\X)_j (z) (\partial_{z_j} f)(z),
\end{equation}
where $z_j$, $j\in\{1,\ldots,\dim(M)\}$, are local coordinates.
Similarly, a smooth time- and parameter-dependent vector field acts as a differential operator
by
\begin{equation}
\label{eqn:vf_derivative_time}(\X  f)(z,t;\eps) = \sum_j (\X)_j (z,t;\eps) (\partial_{z_j} f)(z).
\end{equation}

If $\Phi^t$ is the flow for the vector field $\X$, then
\begin{equation*}\begin{split}\frac{d}{dt}(f(\Phi^t(z)))&=\nabla f(\Phi^t(z))\cdot \frac{d}{dt}(\Phi^t(z))=\nabla f(\Phi^t(z))\cdot\X(\Phi^t(z))\\
&=\sum_{j}(\X_j)(\Phi^t(z))(\partial_{z_j} f)(\Phi^t(z))=(\X f)(\Phi^t(z)).
\end{split}\end{equation*}

For a vector-valued function $\bF:M\to \R^k$   of components $\bF=(\bF_i)_i$, we will denote
\begin{equation*}
\X\bF:=(\X\bF_i)_i.
\end{equation*}

\subsection{Extended system}\label{sec:extended}
To \eqref{eqn:generalperturbation} we associate the extended system
\begin{equation}\label{eqn:generalperturbation_t}
\begin{split}
& \dot z= \X^0(z) + \eps\X^1(z,t;\eps), \\
& \dot t  = 1,\\
\end{split}
\end{equation}which is defined on the extended phase space $\tM=M\times \R$. We denote $\tz=(z,t)\in \tM$.
The independent variable will be denoted by $s$ from now on, and the derivative above is meant with respect to $s$. We will denote by $\tPhi^s_\eps$ the extended flow of \eqref{eqn:generalperturbation_t}. We have
\[ \tPhi^s_\eps(z,t)=( {\Phi}^s_\eps(z), t+s).\]


\subsection{Normally hyperbolic invariant manifolds}\label{sec:NHIM}
We briefly recall the notion of a normally hyperbolic invariant manifold \cite{Fenichel74,HirschPS77}.

Let $M$ be a $C^r$-smooth manifold, $\Phi^t$ a $C^r$-flow on $M$. A submanifold  (with or without boundary) $\Lambda$ of $M$ is a  normally hyperbolic invariant manifold (NHIM) for $\Phi^t$ if  it is invariant under $\Phi^t$, and there exists a splitting of the tangent bundle of $TM$ into sub-bundles over $\Lambda$
\begin{equation}
\label{eqn:NHIM_splitting}
T_z M=E^{\un}_z \oplus E^{\st}_z \oplus T_z \Lambda, \quad \forall z \in \Lambda
\end{equation}
that are invariant under $D\Phi^t$ for all $t\in\mathbb{R}$, and there exist  rates
\[\lambda_-\le \lambda_+<\lambda_c<0<\mu_c<\mu_-\le \mu_+\]
and a constant ${C}>0$, such that for all $x\in\Lambda$ we have
\begin{equation}
\label{eqn:NHIM_rates}
\begin{split} {C}e^{t\lambda_- }\|v\| \leq \|D\Phi^t(z)(v)\|\leq  {C}e^{t\lambda_+}\|v\|  \textrm{ for all } t\geq 0, &\textrm{ if and only if } v\in E^{\st}_z,\\
{C}e^{t\mu_+ }\|v\|\leq \|D\Phi^t(z)(v)\|\leq  {C}e^{t\mu_- }\|v\|  \textrm{ for all } t\leq 0,  &\textrm{ if and only if }v\in E^{\un}_z,\\
{C}e^{|t|\lambda_c }\|v\|\leq  \|D\Phi^t(z)(v)\|\leq  {C}e^{|t|\mu_c}\|v\| \textrm{ for all } t\in\mathbb{R}, &\textrm{ if and only if }v\in T_z\Lambda.
\end{split}
\end{equation}

It is known that  $\Lambda$ is $C^{\ell}$-differentiable, with $\ell\leq r-1$,   provided that
\begin{equation}\label{eqn:ratesdifferentiable}
\begin{split}
& \ell {\mu}_c  + {\lambda}_+ < 0, \\
& \ell {\lambda}_c +  {\mu}_- > 0.
\end{split}
\end{equation}

The manifold  $\Lambda$ has associated  unstable and stable manifolds of $\Lambda$,
denoted $W^{\un}(\Lambda)$ and $W^{\st}(\Lambda)$, respectively,
which are $C^{\ell-1}$-differentiable.
They are foliated by $1$-dimensional unstable and stable manifolds (fibers) of points,
$W^{\un}(z)$, $W^{\st}(z)$, $z\in\Lambda$,  respectively, which are as smooth as the flow,
i.e., $C^r$-differentiable.
These fibers are not invariant by the flow, but \emph{equivariant} in the sense that
\begin{equation*}\begin{split}
\Phi^t(W^\un(z))&=W^\un(\Phi^t(z)),\\
\Phi^t(W^\st(z))&=W^\st(\Phi^t(z)).
\end{split}\end{equation*}

The  unstable and stable manifolds of $\Lambda$,
$W^\un(\Lambda)$ and $W^\st(\Lambda)$, are tangent to
\[E^\un_{\Lambda}=\bigcup_{z\in\Lambda}E^\un_z,  \textrm { and }
E^\st_{\Lambda}=\bigcup_{z\in\Lambda}E^\st_z,\]
respectively.

Since $W^{\st,\un}(\Lambda)=\bigcup_{z\in\Lambda} W^{s,u}(z)$, we can define the projections along the fibers
\begin{equation}\label{eqn:projections}\begin{split}
\Omega^{+}:W^{\st}(\Lambda)\to\Lambda,\quad &\Omega^+(z)=z^+\textrm{ iff }
z \in W^{\st}(z^{+}),\\
\Omega^{-}:W^{\un}(\Lambda)\to\Lambda,\quad &\Omega^-(z)=z^-\textrm{ iff }z \in W^{\un}(z^{-}).\end{split}\end{equation}

The  point
$z^+\in \Lambda$ is characterized by
\begin{equation}\label{eqn:convergence_s}
d( {\Phi}^t(z), {\Phi}^t(z^+) ) \le
C_z e^{t \lambda_+}, \quad \textrm { for all } t \ge
0.
\end{equation}
and the point $z^- \in \Lambda$ by
\begin{equation}\label{eqn:convergence_u}
d(  {\Phi}^t(z), {\Phi}^t(z^-) \le
  C_z e^{t \mu_-}, \quad \textrm { for all } t \leq 0,
\end{equation}
for some $C_z>0$.

\subsection{Scattering map}\label{section:scattering_review}

Assume that $W^\un(\Lambda)$, $W^\st(\Lambda)$ have a   transverse intersection along a manifold $\Gamma$ satisfying:
\begin{equation}\label{goodtransversal}
\begin{split}
&T_z\Gamma = T_xW^\st(\Lambda)\cap
T_xW^\un(\Lambda), \textrm { for all } z\in\Gamma,\\
&T_xM=T_z\Gamma  \oplus
T_xW^\un(z^-)\oplus  T_xW^\st(z^+), \textrm { for all } z\in\Gamma.
\end{split}
\end{equation}

Under these conditions the projection mappings  $\Omega^\pm$ restricted to $\Gamma$ are local diffeomorphisms. We can restrict $\Gamma$ if necessary so that $\Omega^\pm$ are diffeomorphisms from $\Gamma$ onto open subsets $U^\pm$ in  $\Lambda$. Such a $\Gamma$ will be called a \emph{homoclinic channel}. 

By definition the scattering map associated to $\Gamma$ is defined as
\[\sigma : U^- \subseteq \Lambda \to U^+ \subseteq
\Lambda,\quad \sigma = \Omega^+ \circ (\Omega^{-})^{-1}.
\]

Equivalently, $\sigma(z^-) = z^+$, provided that
$W^\un(z^-)$ intersects $W^\st(z^+)$ at a unique point
$z\in\Gamma$.

If $M$ is a  symplectic manifold, $\Phi^t$ is a Hamiltonian flow on $M$, and $\Lambda\subseteq M$ has an induced symplectic structure,
then the scattering map is symplectic. If the flow is exact Hamiltonian,
the scattering map is exact symplectic. For details see \cite{DelshamsLS08a}.

\subsection{Normally hyperbolic invariant manifold for the unperturbed  rotator-pendulum system}
\label{section:NHIM_ unperturbed rotator_pendulum}

Consider the unperturbed rotator-pendulum system described in Section \ref{section:rotator_pendulum}.

The point $(0,0)$ is a hyperbolic fixed point for each pendulum,  the  characteristic
exponents are $\lambda^\un_i=(-V_i''(0))^{1/2}$, $\lambda^\st_i=-(-V_i''(0))^{1/2}$,
 and the corresponding unstable/stable eigenspaces are $E^\un_i=\textrm{Span}(v^u_i)$,
 $E^\st_i=\textrm{Span}(v^s_i)$, where
$v^u_i=(-(-V_i''(0))^{1/2},1)$, $v^s_i=((-V_i''(0))^{1/2},1)$, for $i=1, \ldots,n$.

Define \begin{equation}\label{eqn:def_lambdas}\begin{split} \lambda_- =&-\max_i \lambda_i,\, \lambda_+ =-\min_i \lambda_i,\\ \mu_- =& \min_i \lambda_i,\, \mu_+ =\max_i \lambda_i,\\ \lambda_c=&-\mu_c,\end{split}\end{equation}
where $\mu_c>0$ is some arbitrarily small positive number.

Also, define
\begin{equation}\label{eqn:def_Eus} \begin{split} E^\un_z=& \oplus_{i=1,\ldots,n} E^\un_i,\\
E^\st_z=& \oplus_{i=1,\ldots,n} E^\st_i.\end{split} \end{equation}

It immediately follows that for each closed $d$-dimensional ball $D\subseteq \R^d$, the set
\begin{equation}
\label{eqn:pendulm_NHIM}
\Lambda_0=\{(p,q,I,\theta)\,|\, I\in D,\, p=q=0\}
\end{equation}
is a NHIM with boundary,
where the rates
$\lambda_-$, $\lambda_+$, $\mu_-$, $\mu_+$,
$\lambda_c$, and $\mu_c$ from Section \ref{sec:NHIM} are
the ones defined by \eqref{eqn:def_lambdas}, and
the unstable and stable spaces $E^\un_z$ and
$E^\st_z$ at $z\in\Lambda_0$  are the ones given by \eqref{eqn:def_Eus}.

\subsection{Coordinate system for the unperturbed  rotator-pendulum system}\label{sec:coordinates-rot-pend}
The pendulum-rotator system is initially given in the coordinates $(p,q,I,\theta)$, and
the NHIM $\Lambda_0$ for this system  is described in the action-angle coordinates $(I,\theta)$.
Let  $\mathscr{N}$ be a neighborhood  of  $W^\un(\Lambda_0)=W^\st(\Lambda_0)$.

We define a new  system of symplectic coordinates\footnote{Symplectic coordinates are coordinates obtained via a change of variables that is a symplectic mapping.}
$(y,x,I,\theta)$ in a neighborhood  $\mathscr{N}'\subseteq \mathscr{N}$ of a disk $\mathscr{D}\subseteq W^\un(\Lambda_0)\setminus\{(0,0)\}=W^\st(\Lambda_0)\setminus\{(0,0)\}$, via the following properties:
\begin{itemize}
\item The coordinates $(I,\theta)$ are the action-angle coordinates for the rotator;
\item $dp\wedge dq=dy \wedge dx $;
\item $z\in\Lambda_0$ if and only if $x(z)=y(z)=0$;
\item $z\in W^\un(\Lambda_0)=W^\st(\Lambda_0)$ if and only if $y(z)=0$;
\item for $z\in\mathscr{N}'$, we have that $y_i=\pm(p_i^2/2+V_i(q_i))$, for $i=1,\ldots,n$.
\end{itemize}

\begin{figure}
\includegraphics[width=0.6\textwidth]{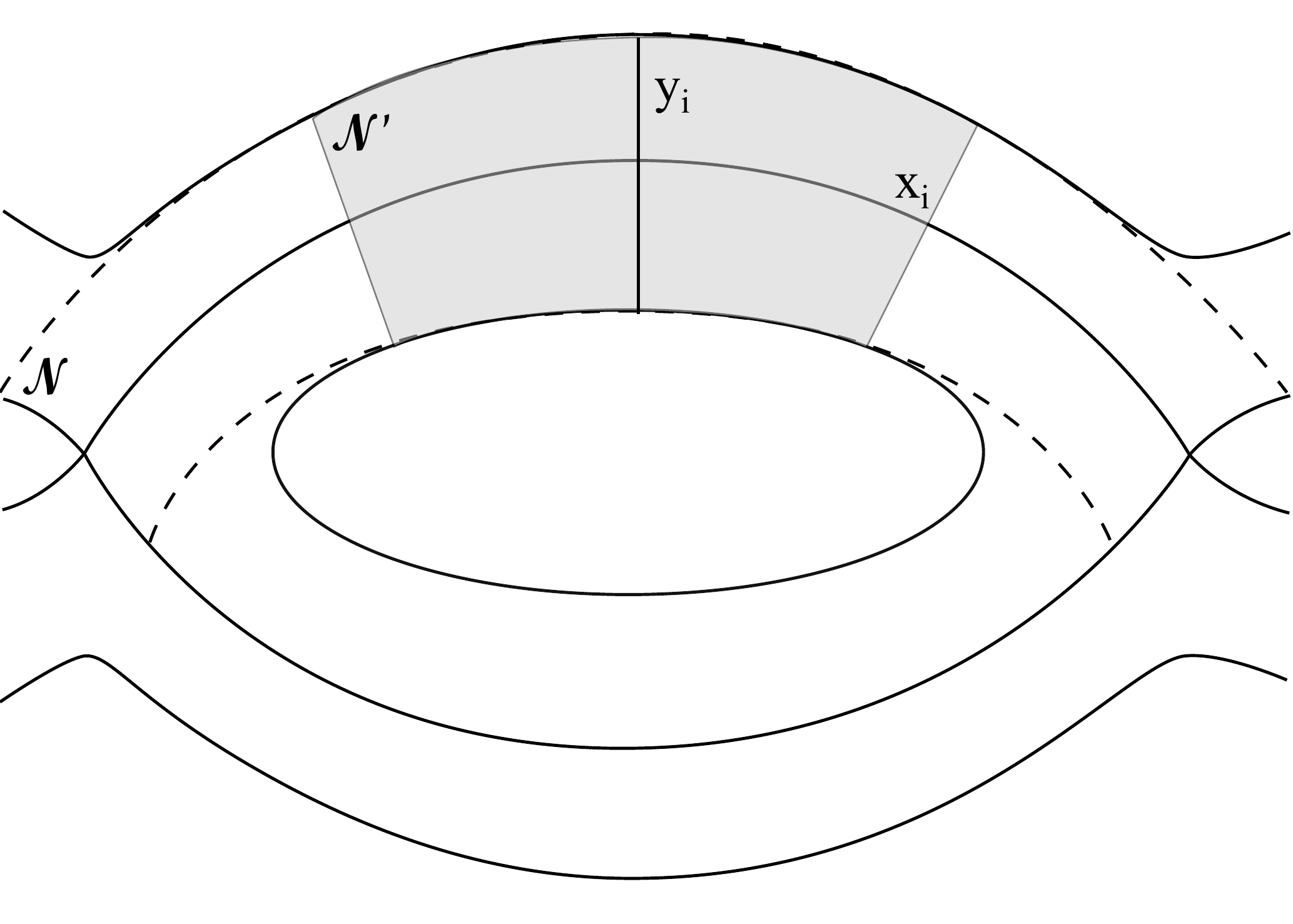}
\caption{Symplectic coordinate system $(y,x)$ in a neighborhood of a segment of the separatrix.}
\label{fig:pendulum_coordinates}
\end{figure}

See Fig.~\ref{fig:pendulum_coordinates}. The coordinate $y_i$ can be chosen to be equal to the energy $\pm(p_i^2/2+V_i(q_i))$ in a whole neighborhood  of the separatrix of the $i$-th generalized pendulum. 

Once we have that $y_i$ is the energy of the $i$-th generalized pendulum, 
the coordinate $x_i$ is determined so that is the symplectic conjugate of $y_i$.

The coordinate $x_i$ is given by $dx_i=\frac{ds_i}{\|\nabla y_i\|}$, where $ds_i=(dp_i^2+dq_i^2)^{1/2}=(p'_i(t)^2+q'_i(t)^2)^{1/2}dt$ is the arc length element along the energy level. 
Since $(p'_i,q'_i)=(-V'_i(q_i),p_i)$ we have $\|(p'_i,q'_i)\|=\|\nabla y_i\|$, therefore  $dx_i=dt$. That is, the coordinate $x_i$ equals to the time $t$ it takes the solution $(p_i(t),q_i(t))$ to go from some initial point $(p_{i}^{0},q_{i}^{0})$ to $(p_i,q_i)$.
The value $q_{i}^{0}$ can be chosen uniformly for all energy levels, and $p_{i}^{0}$ is implicitly given by the energy condition.

A direct computation shows that
\[ dx_i=\frac{ds}{\|y_i\|}=\frac{-V'_i(q_i)dp_i+p_idq_i}{p_i^2+V'_i(q_i)^2},\]
hence
\[dy_i\wedge dx_i=(p_idp_i+V'_i(q_i)dq_i)\wedge \left(\frac{-V'_i(q)dp_i+p_idq_i}{p_i^2+V'_i(q_i)^2}\right)=dp_i\wedge dq_i.\]

Note that we cannot extend  $(y_i,x_i)$ as a symplectic coordinate system to a neighborhood of the separatrix that contains the equilibrium point of the generalized pendulum, since this is a critical point of the energy function.

In the new coordinates $(y,x,I,\theta)$ the Hamiltonian $H_0$ is given by
\begin{equation}\label{eqn:H0_x_y_coords}
H_0(y,x,I,\theta)=h_0(I)+h_1(y)=h_0(I)+y,\textrm { for }  (y,x,I,\theta)\in\mathcal{N}'.
\end{equation}

The coordinate system described above is essentially the same as in\break \cite{gidea2018global}, except that here we additionally emphasize that it is symplectic.

\subsection{The scattering map for the unperturbed  extended pendulum-rotator system}
\label{sec:scattering_rot_pend}

We consider the extended  system from Section \ref{sec:extended}, and we express the scattering map for the  unperturbed  extended pendulum-rotator system in terms of the action-angle coordinates defined in Section~\ref{sec:coordinates-rot-pend}.

Since we have $W^\st(\tilde{\Lambda}_0)=W^\un(\tilde{\Lambda}_0)$ and for each $\tz\in\tilde{\Lambda}_0$, $W^\st(\tz)=W^\un(\tz)$, the corresponding scattering map $\tilde{\sigma}_0$ is the identity map wherever it is defined. Thus, $\tilde{\sigma}_0(\tz^-)=\tz^+$ implies $\tz^-=\tz^+$, or, equivalently
\begin{equation}\label{eqn:sigma0i}\tilde{\sigma}_0(I,\theta,t)=(I,\theta,t).\end{equation}

\subsection{Evolution equations}\label{sec:evolution}
Consider  the coordinate system  $(y,x,I,\theta)$ defined in Section~\ref{sec:coordinates-rot-pend}.
We will identify the vector fields $\X^0$ and $\X^1$ with derivative operators acting on functions, as described in Section \ref{section:notation}.

Since $\X^0=J\nabla H_0$ is a Hamiltonian vector field, using the Poisson bracket $[\cdot,\cdot]$, we have
\begin{equation*}\begin{split}
\X^0y     = &[y,H_0]=[y,h_0(I)+h_1(y,x)]=-\frac{\partial h_1}{\partial x},\\
\X^0x     = &[x,H_0]=[x,h_0(I)+h_1(y,x)]=\frac{\partial h_1}{\partial y},\\
\X^0I     = &[I,H_0]=[I,h_0(I)+h_1(y,x)]=-\frac{\partial h_0}{\partial \theta}=0,\\
\X^0\theta= &[\theta,H_0]=[\theta,h_0(I)+h_1(y,x)]=\frac{\partial h_0}{\partial I}=\omega(I).
\end{split}\end{equation*}

When $\X^1=J\nabla H_1$ is a Hamiltonian vector field, similarly we have
\begin{equation*}\begin{split}
\X^1y          = &[y,H_1] =  -\frac{\partial H_1}{\partial x},\\
\X^1x          = &[x,H_1] =   \frac{\partial H_1}{\partial y},\\
\X^1I          = &[I,H_1] =  -\frac{\partial H_1}{\partial \theta},\\
\X^1\theta     = &   [\theta,H_1]=\frac{\partial H_1}{\partial I}.
\end{split}\end{equation*}

Using the above formulas,  we provide below the evolution equations of the  coordinates $(y,x,I,\theta)$, expressing the time-derivative of each coordinate along a solution of the perturbed system.
We include the expression for the general case, as well as for the special case when the perturbation $\X^1$ is Hamiltonian.
\begin{equation}\label{eqn:evolution2}
\begin{split}
\dot y &=   \X^0y + \eps\X^1y  = -\frac{\partial{H_0}}{\partial x} + \eps\X^1y\\&=
-\frac{\partial{h_1}}{\partial x}-\eps\frac{\partial{H_1}}{\partial x}.
\end{split}
\end{equation}
\begin{equation}\label{eqn:evolution1}
\begin{split}
\dot x &=   \X^0x + \eps\X^1x  =\frac{\partial{H_0}}{\partial y} + \eps\X^1x\\&
=\frac{\partial{h_1}}{\partial y}+\eps\frac{\partial{H_1}}{\partial y}.
\end{split}
\end{equation}

\begin{equation}\label{eqn:evolution3}
\begin{split}
\dot I &= \X^0I+\eps \X^1I =  -\frac{\partial{H_0}}{\partial \theta} + \eps\X^1I\\&=
-\eps \frac{\partial{H_1}}{\partial \theta}.
\end{split}
\end{equation}
\begin{equation}
\label{eqn:evolution4}
\begin{split}\dot \theta &=  \X^0\theta+\eps \X^1 \theta =  \frac{\partial{H_0}}{\partial I} + \eps\X^1\theta \\&=
 \frac{\partial{h_0}}{\partial I}+\eps \frac{\partial{H_1}}{\partial I}.
\end{split}
\end{equation}

Note that the evolution equations for the $x$- and $y$-coordinate from above are only valid for $(y,x,I,\theta)\in\mathcal{N}'$ from Section~\ref{sec:coordinates-rot-pend}.

\subsection{Perturbed normally hyperbolic invariant manifolds}\label{section:NHIM_perturbed}
Since $\Lambda_0$ is a NHIM for the flow $\Phi^t_0$ of $\X^0$,
$\tLambda_0=\Lambda_0\times \R$ is a NHIM for the flow $\tPhi^s_0$ of the extended system \eqref{eqn:generalperturbation_t}.

Recall that  $\X^1=\X^1(z,t,\eps)$ is assumed to be
uniformly differentiable in all variables.  The theory of normally hyperbolic invariant manifolds,
\cite{Fenichel71,HirschPS77,Pesin04} asserts that there exists $\eps_0$ such that the manifold $\tLambda_0$
persists as a  normally hyperbolic manifold $\tLambda_\eps$, for all $|\eps| < \eps_0$, which is locally invariant under the flow $\tPhi^t_\eps$.
The persistent NHIM $\tLambda_\eps$ is $O(\eps)$-close in the $C^\ell$-topology to $\tLambda_0$, where $\ell$ is as in \eqref{eqn:ratesdifferentiable}.  The locally invariant manifolds are in fact invariant manifolds for an extended system, and they depend on the extension. Hence,  they do not need to be unique.

The manifold $\tLambda_\eps$ can be parametrized
via a $C^{\ell}$-diffeomorphism $\tilde{k}_\eps:\tLambda_0\to\tLambda_\eps$, where $\tilde{k}_0=\textrm{Id}_{\tLambda_0}$, and $\tilde{k}_\eps$  is $O(\eps)$-close to $\tilde{k}_0$ in the $C^\ell$-smooth topology on compact sets. Through $\tilde{k}_\eps$, the perturbed NHIM $\tLambda_\eps$ can be parametrized in terms of the variables $(I,\theta,t)$, where $(I,\theta)$ are the action-angle variables on $\Lambda_0$.

For details, see \cite{DelshamsLS06a}.

For the  perturbed NHIM $\tLambda_\eps$, $|\eps| < \eps_0$, there exists an invariant splitting of the tangent bundle similar  to that in \eqref{eqn:NHIM_splitting}, and $D\tPhi^t_\eps$ satisfies expansion/contraction relations similar to those in
\eqref{eqn:NHIM_rates}, for some constants $\tilde C$, $\tilde\lambda_-$, $\tilde\lambda_+$, $\tilde\mu_-$, $\tilde\mu_+$, $\tilde\lambda_c$, $\tilde\mu_c$.
These constants are independent of $\eps$, and can be chosen as close as desired to the unperturbed ones, that is, to
$C$, $\lambda_-$, $\lambda_+$,  $\mu_-$, $\mu_+$, $\lambda_c$, $\mu_c$, respectively,  by choosing $\eps_0$ suitably small.

There exist unstable and stable manifolds $W^\un(\tLambda_\eps)$, $W^\st(\tLambda_\eps)$ associated to $\tLambda_\eps$,
and there exist corresponding projection maps $\Omega^-:W^\un(\tLambda_\eps)\to\tLambda_\eps$, and $\Omega^+: W^\st(\tLambda_\eps)\to\tLambda_\eps$.
For $\tilde{z}_+=\Omega^+(\tilde{z})$, with $\tilde{z}\in W^\st(\tLambda_\eps)$ we have
\begin{equation}\label{eqn:convergence_s}
d( {\Phi}^t(\tilde{z}), {\Phi}^t(z^+) ) \le
 C_{\tilde{z}} e^{t \tilde{\lambda}_+}, \quad \textrm { for all } t \ge
0.
\end{equation}
and for $\tilde{z}_-=\Omega^-(\tilde{z})$, with $\tilde{z}\in W^\un(\tLambda_\eps)$ we have
\begin{equation}\label{eqn:convergence_u}
d(  {\Phi}^t(\tilde{z}), {\Phi}^t(\tilde{z}^-) \le
  C_{\tilde{z}} e^{t \tilde{\mu}_-}, \quad \textrm { for all } t \leq 0,
\end{equation}
for some $\tilde{C}_{\tilde{z}}>0$. The constant $\tilde{C}_{\tilde{z}}$ can be chosen uniformly bounded provided we restricted to $\tilde{z}$ in the local unstable and stable manifolds  $W^\un_{\rm loc}(\tLambda_\eps)$, $W^\st_{\rm loc}(\tLambda_\eps)$. Hence we can replace  $\tilde{C}_{\tilde{z}}$ by some $\tilde{C}$.

To simplify notation, from now on we will drop the symbol $\tilde{}$ from $\tilde C$, $\tilde\lambda_-$, $\tilde\lambda_+$,  $\tilde\mu_-$, $\tilde\mu_+$, $\tilde\lambda_c$ , $\tilde\mu_c$.

\section{Master lemmas}\label{sec:master_lemmas}

In this section we define some abstract Melnikov-type integral operators and study their properties, which will be used in the next sections. The derivations are similar to the ones in \cite{gidea2018global}. 

From Section \ref{section:NHIM_perturbed}, there exists $\eps_0>0$ such that, for each $\eps\in(-\eps_0,\eps_0)$, there exists a normally hyperbolic invariant manifold $\tLambda_\eps$ for $\tPhi^s_\eps$.

Assume that for each $\eps\in(-\eps_0,\eps_0)$ there exists  a homoclinic channel  $\tGamma_\eps$ (see Section \ref{section:scattering_review}), which depends $C^{\ell}$-smoothly on $\eps$, and determines the projections $\Omega^\pm:\tGamma_\eps\to \Omega^\pm(\tGamma_\eps)\subseteq \tLambda_\eps$, which are local diffeomorphisms as in \eqref{eqn:projections}.
We are thinking of $\tPhi^s_\eps$, $\tLambda_\eps$, $\tGamma_\eps$  as perturbations of
$\tPhi^t_0$, $\tLambda_0$, $\tGamma_0$, for $\eps\neq 0$ small.

Let $\tz_\eps\in\tGamma_\eps$ be a homoclinic point for $\tPhi^s_\eps$.
Because of the smooth dependence
of the normally hyperbolic manifold  and of its stable and unstable
manifolds on the perturbation,  there is a homoclinic point
$\tz_0\in\tGamma_0$   for $\tPhi^s_0$  that is $O(\eps)$-close to $\tz_\eps$, that is
\begin{equation}\label{eqn:zeps_z0}
\tz_\eps=\tz_0+O(\eps).
\end{equation}

Let $(\tz_\eps,\eps)\in\widetilde{M}\mapsto {\bf F}(\tz_\eps,\eps)\in\mathbb{R}^k$ be a uniformly  $C^{r_0}$-smooth  mapping on  $\widetilde{M}\times\R$, with $1\leq r_0\leq r$.

We define the integral operators
\begin{equation}\label{eqn:master_operators}
\begin{split}
\mathfrak{I}^+(\bF,\Phi^s_\eps, \tz_\eps)=&\int_{0}^{+\infty} \left( \mathbf{F}(\Phi^s_\eps(\tz_\eps^+))-\mathbf{F}(\Phi^s_\eps(\tz_\eps))\right)ds,\\
\mathfrak{I}^-(\bF,\Phi^s_\eps,\tz_\eps)=&\int_{-\infty}^{0} \left( \mathbf{F}(\Phi^s_\eps(\tz^-_\eps))-\mathbf{F}(\Phi^s_\eps(\tz_\eps)\right)ds.
\end{split}\end{equation}

\begin{lem}[Master Lemma 1]\label{lem:master_1}
The  improper integrals \eqref{eqn:master_operators} are convergent.
The operators $\mathfrak{I}^+(\bF,\Phi^s_\eps, z_\eps)$ and $\mathfrak{I}^-(\bF,\Phi^s_\eps, z_\eps)$
are linear in $\bF$.
\end{lem}
\begin{proof}
The linearity of the operators  follows from the linearity properties of integrals.

To prove convergence, we will use that the  exponential contraction along the stable (unstable) manifold in forward (backward) time, given by\eqref{eqn:convergence_s} and  \eqref{eqn:convergence_u}.
For the stable  manifold, we have
\[|\tPhi^{s}(\tz^{+}_{\varepsilon})-\tPhi^{s}(\tz_{\varepsilon})|<Ce^{s \lambda_{+}},\textrm{ for } s\geq 0,\]
where $C$ is the positive constant and $\lambda_{+}$ is the negative contraction rate from Section \ref{section:NHIM_perturbed}.

Recall that $\textbf{F}$ is uniformly $\mathcal{C}^{r_0}$-differentiable,
so it is Lipschitz with Lipschitz constant $C$. Thus,
\begin{eqnarray*}
|\mathfrak{J}^{+}(\textbf{F},\Phi^{s}_{\varepsilon},z_{\varepsilon})|
&=& \left|\int^{\infty}_{0} \textbf{F}(\Phi^{s}_{\varepsilon}(z_{\varepsilon}^{+}))-\textbf{F}(\Phi^{s}_{\varepsilon}(z_{\varepsilon})) ds\right|\\
&\leq& \int^{\infty}_{0}C_{\textbf{F}}C_{z}e^{s \lambda_{+}} ds\\
&=&-C_{\textbf{F}}C_{z}\frac{1}{\lambda_{+}}
\end{eqnarray*}
Note, the last expression is positive since $\lambda_{+}<0$. Thus the integral is bounded and therefore convergent. The proof for the convergence of
$\mathfrak{J}^{-}(\textbf{F},\Phi^{s}_{\varepsilon},z_{\varepsilon})$ is similar. The difference is that the limits of integration are from $-\infty$ to $0$
and the contraction rate is $-\mu_{-}<0$

Also, the proof holds if $\textbf{F}$ is replaced by any Lipschitz function, in particular, by $\X_\eps \textbf{F}$, where we recall that $\X_\eps=\X_0+\eps\X^1$.  This fact will be used in the proof of the next lemma.
\end{proof}

\begin{lem}[Master Lemma 2]\label{lem:master_2}
\begin{equation}\label{eqn:master_differences}
\begin{split}
\bF(\tz^+_\eps)-\bF(\tz_\eps)=&-\mathfrak{I}^+((\X^0+\eps\X^1)\bF,\tPhi^s_\eps, \tz_\eps),\\
\bF(\tz^-_\eps)-\bF(\tz_\eps)=&\mathfrak{I}^-((\X^0+\eps\X^1)\bF,\tPhi^s_\eps, \tz_\eps).
\end{split}
\end{equation}
\end{lem}
\begin{proof}
To prove this lemma, we will begin by computing the derivative of the $i^{\text{th}}$ component of $\textbf{F}$ along the perturbed flow. For $\tz$   a point in $\widetilde{M}$, using \eqref{eqn:vf_derivative_time} we have
\begin{eqnarray*}
\frac{d}{ds}\textbf{F}_{i}(\tPhi^{s}_{\varepsilon}(\tz)) &=& \nabla \textbf{F}_{i}(\tPhi^{s}_{\varepsilon}(\tz))\cdot \frac{d}{ds}\tPhi^{s}_{\varepsilon}(\tz) \\
&=&\X^{0}\textbf{F}_{i}(\tPhi^{s}_{\varepsilon}(\tz))+\varepsilon \X^{1}\textbf{F}_{i}(\tPhi^{s}_{\varepsilon}(\tz);\eps).
\end{eqnarray*}

With the above result, we can now compute the difference in \eqref{eqn:master_differences}. Note that we define a vector field, $\X$, acting on a vector valued function, $\textbf{F}$, as
\[\chi \textbf{F}=(\chi \textbf{F}_{i})_{i}.\]

We have
\begin{eqnarray*}
\textbf{F}(\tz^{+}_{\varepsilon})-\textbf{F}(\tz_{\varepsilon})
&=&\textbf{F}(\tPhi^{T}_{\varepsilon}(\tz_{\epsilon}^{+}))-\textbf{F}(\tPhi^{T}_{\varepsilon}(\tz_{\epsilon})) \\  
&\quad&-\int^{T}_{0}\frac{d}{ds}\left( \textbf{F}(\tPhi^{s}_{\varepsilon}(\tz_{\epsilon}^{+}))-\textbf{F}(\tPhi^{s}_{\varepsilon}(\tz_{\epsilon}))\right)ds.
\end{eqnarray*}

Letting $T$ approach infinity, the first difference vanishes because the homoclinic point $\tz_{\epsilon}$ and its foot point $\tz_{\epsilon}^{+}$ approach each other. We then can rewrite the integral using the expression for the derivative of $\textbf{F}$
along the flow:
\begin{equation*}\begin{split}
&-\int^{\infty}_{0}\left(\X^{0}\textbf{F}(\tPhi^{s}_{\varepsilon}(\tz_{\epsilon}^{+}))-
\X^{0}\textbf{F}(\tPhi^{s}_{\varepsilon}(\tz_{\epsilon}))\right)ds\\&\qquad-
\varepsilon\int^{\infty}_{0}\left(\X^{1}\textbf{F}(\tPhi^{s}_{\varepsilon}(\tz^{+}_{\epsilon});\eps)-
\X^{1}\textbf{F}(\tPhi^{s}_{\varepsilon}(\tz_{\epsilon});\eps)\right)ds \\
&=-\int^{\infty}_{0}\left(\X^{0}\textbf{F}(\tPhi^{s}_{\varepsilon}(\tz_{\epsilon}^{+}))
+\eps\X^{1}\textbf{F}(\tPhi^{s}_{\varepsilon}(\tz^{+}_{\epsilon});\eps)\right)ds\\
&\qquad -\int^{\infty}_{0}\left(\X^{0}\textbf{F}(\tPhi^{s}_{\varepsilon}(\tz_{\epsilon}))+
\X^{1}\textbf{F}(\tPhi^{s}_{\varepsilon}(\tz_{\epsilon});\eps)\right)ds
\\
&=-\int^{\infty}_{0}\left((\X^{0}+\varepsilon\X^{1})\textbf{F}(\tPhi^{s}_{\varepsilon}(\tz_{\epsilon}^{+}))
-(\X^{0}+\varepsilon\X^{1})\textbf{F}(\tPhi^{s}_{\varepsilon}(\tz_{\epsilon}))\right) ds
\\
&=-\mathfrak{J}^{+}((\X^{0}+\varepsilon \X^{1})\textbf{F},\tPhi^{s}_{\varepsilon},\tz_{\varepsilon}).
\end{split}\end{equation*}

The proof for $\mathfrak{J}^{-}((\X^{0}+\varepsilon \X^{1})\textbf{F},\tPhi^{s}_{\varepsilon},\tz_{\varepsilon})$ is similarly. The main difference is that the limits of integration are from $-\infty$ to $0$.
\end{proof}

\begin{lem}[Master Lemma 3]\label{lem:master_3}
\begin{equation}\label{eqn:master_gronwall}
\begin{split}
\mathfrak{I}^+(\bF,\tPhi^s_\eps, \tz_\eps)=&\mathfrak{I}^+(\bF,\tPhi^s_0, \tz_0)+O(\eps^{\varrho}),\\
\mathfrak{I}^-(\bF,\tPhi^s_\eps, \tz_\eps)=&\mathfrak{I}^-(\bF,\tPhi^s_0, \tz_0)+O(\eps^{\varrho}),
\end{split}
\end{equation}
for $0<\varrho<1$. The integrals on the right-hand side are evaluated with $\X^1=\X^1(\cdot;0)$. 
\end{lem}
\begin{proof}

To prove this lemma, we will use both the Gronwall inequality from the Appendix \ref{sec:gronwall} and the Lipschitz property of $\textbf{F}$.
The Gronwall inequality \eqref{eqn:eq_4} gives

\[\tPhi^{s}_{\varepsilon}(\tz^{+}_{\varepsilon})=\tPhi^{s}_{0}(\tz^{+}_{0})+O(\varepsilon^{\rho_{1}}),\]
and
\[\tPhi^{s}_{\varepsilon}(\tz_{\varepsilon})=\tPhi^{s}_{0}(\tz_{0})+O(\varepsilon^{\rho_{1}}),\]
where $0<\rho_{1}<1$. Note that these equalities hold on an interval of time $0<t<k\ln\left(\frac{1}{\varepsilon}\right)$, for $k\leq\frac{1-\rho}{C_0}$, where $C_0$ is the Lipschitz constant of $\X_0$; see Appendix \ref{sec:gronwall}.

Before using the results from Gronwall, we will split the integrals into two parts:
\begin{eqnarray*}
\mathfrak{J}^{+}(\textbf{F},\tPhi^{s}_{\varepsilon},\tz_{\varepsilon})-\mathfrak{J}^{+}(\textbf{F},\tPhi^{s}_{0},\tz_{0}))
&=&\,\int_{0}^{\infty}\textbf{F}(\tPhi^{s}_{\varepsilon}(\tz_{\epsilon}^{+}))
-\textbf{F}(\tPhi^{s}_{\varepsilon}(\tz_{\epsilon}))ds \\
&&-\int_{0}^{\infty}\textbf{F}(\tPhi^{s}_{0}(\tz_{0}^{+}))-\textbf{F}(\tPhi^{s}_{0}(\tz_{0}))ds.
\end{eqnarray*}

Consider the first integral, which can be written as
\[\int_{0}^{T}\textbf{F}(\tPhi^{s}_{\varepsilon}(\tz_{\epsilon}^{+}))-
\textbf{F}(\tPhi^{s}_{\varepsilon}(\tz_{\epsilon}))ds+
\int_{T}^{\infty}\textbf{F}(\tPhi^{s}_{\varepsilon}(\tz_{\epsilon}^{+}))-
\textbf{F}(\tPhi^{s}_{\varepsilon}(\tz_{\epsilon}))ds.\]

Examining the second of these two integrals, we have
\begin{eqnarray*}
\left|\int_{T}^{\infty}\textbf{F}(\tPhi^{s}_{\varepsilon}(\tz_{\epsilon}^{+}))-
\textbf{F}(\tPhi^{s}_{\varepsilon}(\tz_{\epsilon}))ds\right|
&\leq& \int^{\infty}_{T}C_{\textbf{F}}Ce^{s \lambda_{+}} ds \\  
&=& {\bf C}\frac{1}{|\lambda_{+}|}e^{T \lambda_{+}}
\end{eqnarray*}
where ${\bf C}=C_{\textbf{F}}C$.

Now if we let $T=k\ln\left(\frac{1}{\varepsilon}\right)$, then the integral is bounded by
$$C\frac{1}{|\lambda_{+}|}\varepsilon^{k|\lambda_{+}|}$$
More importantly, we have shown that the integral is bounded by $O\left(\varepsilon^{\rho_{2}}\right)$ with $\rho_{2}=k|\lambda_{+}|$.
 
A similar argument holds for
\[\int_{T}^{\infty}\textbf{F}(\tPhi^{s}_{0}(\tz_{0}^{+}))-\textbf{F}(\tPhi^{s}_{0}(\tz_{0}))ds.\]
Returning to the integral from $0$ to $T$, we have
\begin{equation*}\begin{split}
\int_{0}^{T}\textbf{F}(\tPhi^{s}_{\varepsilon}(\tz_{\epsilon}^{+}))-\textbf{F}(\tPhi^{s}_{\varepsilon}(\tz_{\epsilon}))ds
-\int_{0}^{T}\textbf{F}(\tPhi^{s}_{0}(\tz_{0}^{+}))-\textbf{F}(\tPhi^{s}_{0}(\tz_{0}))ds
\\= \int_{0}^{T}\textbf{F}(\tPhi^{s}_{\varepsilon}(\tz_{\epsilon}^{+}))- \textbf{F}(\tPhi^{s}_{0}(\tz_{0}^{+}))ds
- \int_{0}^{T}\textbf{F}(\tPhi^{s}_{\varepsilon}(\tz_{\epsilon}))-\textbf{F}(\tPhi^{s}_{0}(\tz_{0}))ds.
\end{split}\end{equation*}

Now we can apply the Gronwall inequality \eqref{eqn:eq_4} as well as the Lipschitz property of $\textbf{F}$. This show that the difference of the integrals is bounded by
\[\int_{0}^{T}C_{\textbf{F}}O(\varepsilon^{\rho_{1}})ds.\]

The order of the integral is bounded by
\[O\left(\varepsilon^{\rho_{1}}\ln\left(\frac{1}{\varepsilon}\right)\right)=O(\eps^{\rho_3}),\]
for some $\rho_{3}<\rho_{1}<1$. 

Finally, let $\rho=\min\{\rho_{2},\rho_{3}\}$. Returning to the original expression, we have
\[|\mathfrak{J}^{+}(\textbf{F},\tPhi^{s}_{\varepsilon},\tz_{\varepsilon})-
\mathfrak{J}^{+}(\textbf{F},\tPhi^{s}_{0},\tz_{0})|
\leq O(\varepsilon^{\rho}).\]
\end{proof}

\begin{lem}[Master Lemma 4]\label{lem:master_4}
If $\|\textbf{F}\|_{C^1}$ is $O(\eps)$ then
\begin{equation}\label{eqn:master_gronwall}
\begin{split}
\mathfrak{I}^+(\bF,\Phi^s_\eps, z_\eps)=&\mathfrak{I}^+(\bF,\Phi^s_0, z_0)+O(\eps^{1+\varrho}),\\
\mathfrak{I}^-(\bF,\Phi^s_\eps, z_\eps)=&\mathfrak{I}^-(\bF,\Phi^s_0, z_0)+O(\eps^{1+\varrho}),
\end{split}
\end{equation}
for $0<\varrho<1$.
The integrals on the right-hand side are evaluated with $\X^1=\X^1(\cdot;0)$. 
\end{lem}
\begin{proof}
By the mean value theorem, we have
\[|\textbf{F}(z_{1})-\textbf{F}(z_{2})|\leq M |z_{1}-z_{2}|,\]
where $M=\sup|D\textbf{F}(z)|$. Note that $M<\infty$ since $\textbf{F}$ is bounded together with its derivatives.  
Now, by the hypothesis, we can bound $M$ by $O(\varepsilon)$. The proof is now similar to the proof of Lemma
4.3. Essentially, the Lipschitz constant of $\textbf{F}$, $C_{\textbf{F}}$, is replaced with $O(\varepsilon)$. Thus,
\[|\mathfrak{J}^{+}(\textbf{F},\Phi^{s}_{\varepsilon},z_{\varepsilon})-
\mathfrak{J}^{+}(\textbf{F},\Phi^{s}_{0},z_{0})|\leq
O(\varepsilon)O(\varepsilon^{\rho_{1}})+O(\varepsilon)O(\varepsilon^{\rho_{2}}).\]
Finally, let $\rho=\min\{\rho_{1},\rho_{2}\}$.

The proof for $|\mathfrak{J}^{-}(\textbf{F},\Phi^{s}_{\varepsilon},z_{\varepsilon})-\mathfrak{J}^{-}(\textbf{F},\Phi^{s}_{0},z_{0})|$ follows similarly.
\end{proof}

\section{Scattering map for the perturbed rotator-pendulum system}
\label{sec:scattering_perturbed}
\subsection{Existence of transverse homoclinic connections}

Consider the coordinate system $z=(y,x,I,\theta)$ defined in Section \ref{sec:coordinates-rot-pend}.
Let us restrict to $z=(y,x,I,\theta)$ in the neighborhood $\mathscr{N}'$  where $y_i=\pm\frac{1}{2}(p_i^1+V_i(q_i))$.
In the unperturbed case, $W^\un(\tLambda_0)=W^\st(\tLambda_0)$ are given by $y=0$.

In terms of the extended coordinates $(y,x,I,\theta,t)$,  a point  $\tz^*_0\in\tLambda_0$
can be described as
\[\tz^*_0=(0,0,I,\theta,t),\]
and applying the flow to this point yields
\[\tPhi^s_0(\tz^*_0)=(0,0,I,\theta+\omega(I)s,t+s).\]

A point $\tz_0\in W^\un(\tLambda_0)=W^\st(\tLambda_0)$ can be described in coordinates as
\[\tz_0=(0,x_0,I,\theta),\]
and applying the flow to this point yields
\[\tPhi^s_0(\tz_0)=(0,x(s),I,\theta+\omega(I)s,t+s),\]
where $x(s)$ represents the $x$-component of the solution curve of the Hamiltonian $h_1$ with initial condition at $s=0$ equal to $(0,x_0)$, evaluated at time~$s$.

In the perturbed case, for $\eps\neq 0$ small, we can describe both the stable and unstable manifolds as graphs of $C^{\ell-1}$-smooth functions $y^\st_\eps$, $y^\un_\eps$,  over
$(x,I,\theta,t)$ given by
\begin{equation}\label{eqn:y_graphs}
\begin{split}
y^\st_\eps=&y^\st_\eps(x,I,\theta,t),\\
y^\un_\eps=&y^\un_\eps(x,I,\theta,t),
\end{split}
\end{equation}
respectively, for $(x,I,\theta,t)\in \mathscr{N}'\cap W^\un(\tLambda_0)$.

The result below gives sufficient conditions for the existence of a transverse homoclinic intersection  of $W^\st(\tLambda_\eps)$ and $W^\un(\tLambda_\eps)$. The proof is essentially the same as for Proposition 2.6. in \cite{gidea2018global}, except that  the latter  is under the assumption that the perturbation is Hamiltonian.
Therefore we will omit the proof.

\begin{thm}\label{thm:transverse_homoclinic}
For $(x,I,\theta,t)\in \mathscr{N}'$ the difference between $y^\st_\eps(x,I,\theta,t)$ and $y^\un_\eps(x,I,\theta,t)$ is given by
\begin{equation}\label{eqn:yun_minus_yst}\begin{split}
y^\un_\eps-y^\st_\eps=&-\eps\int_{-\infty}^{+\infty}\left(\X^{1}y
\left(0,0,I,\theta+\omega(I)s,t+s\right)\right.\\
  &\qquad\qquad\left.-\X^{1}y\left(0,x(s),I,\theta+\omega(I)s,t+s\right)\right)ds\\
  &+O\left(\eps^{1+\rho}\right)\\
= &-\eps\int_{-\infty}^{+\infty}\left([y,H_1]\left(0,0,I,\theta+\omega(I)s,t+s\right)\right.\\
  &\qquad\qquad \left.-[y,H_1]\left(0,x(s),I,\theta+\omega(I)s,t+s\right)\right)ds
  \\
  &+O\left(\eps^{1+\rho}\right).
\end{split}
\end{equation}
The second formula  corresponds to the case when the perturbation is Hamiltonian.

If $x^*=x^*(I,\theta,t)$ is a non-degenerate zero of the mapping
\begin{equation}\label{eqn:yun_minus_yst_potential}\begin{split}
x_0\in\R^n\mapsto - \int_{-\infty}^{+\infty}&\left(\X^{1}y\left(0,0,I,\theta+\omega(I)s,t+s\right)\right.\\
  & \left. -\X^{1}y\left(0,x(s),I,\theta+\omega(I)\sigma,t+s\right)\right)ds \in \R,
\end{split}
\end{equation}
then there exists $\eps_0>0$ sufficiently small such that for all $0<|\eps|<\eps_0$
$W^\st(\tLambda_\eps)$ and $W^\un(\tLambda_\eps)$ have a transverse homoclinic intersection which can be parametrized as
\[y^\un_\eps(x^*(I,\theta,t),I,\theta,t)=y^\st_\eps(x^*(I,\theta,t),I,\theta,t),\]
for $(I,\theta,t)$ in some open set in $\R^d\times \T^d\times \T^1$.
\end{thm}
\begin{rem}
In the case when  both the system and the perturbation are Hamiltonian, it is shown in \cite{gidea2018global} that the corresponding condition \eqref{eqn:yun_minus_yst_potential} is   for a $C^1$-open and $C^\infty$dense set of perturbation $H_1$. In particular, it is generic.

When the perturbation is non-conservative, it is possible that $W^\st(\tLambda_\eps)$ and $W^\un(\tLambda_\eps)$ do not intersect for any $\eps\neq 0$, even though for $\eps=0$ we have $W^\st(\tLambda_0)=W^\un(\tLambda_0)$.
That is, a non-conservative perturbation can destroy the homoclinic intersection.
The condition \eqref{eqn:yun_minus_yst_potential} that guarantees the existence of such an intersection is non-generic. The next sections are under the assumption that
$W^\st(\tLambda_\eps)$ and $W^\un(\tLambda_\eps)$  intersect transversally for $0<|\eps|<\eps_0$.
\end{rem}

\subsection{Change in action by the scattering map}
\label{sec:pend_rot_change_action}
{$ $}

Assume: \begin{itemize}
\item $\tz_\eps$ is a homoclinic point for the perturbed system, i.e., $\tz_\eps\in
W^\st(\tLambda_\eps)\cap W^\un(\tLambda_\eps)$, 
\item$\tz^\pm_\eps=\Omega^\pm(\tz_\eps)\in\tLambda_\eps$,
\item $\tz_0$ is a homoclinic point for the unperturbed system, i.e.,  $\tz_0\in W^\st(\tLambda_0)\cap W^\un(\tLambda_0)$,
corresponding to $\tz_0$ via \eqref{eqn:zeps_z0}, and 
\item $\tz^\pm_0=\Omega^\pm(\tz_0)\in\tLambda_0$.
\end{itemize}

The existence of the homoclinic point $\tz_\eps$ is guaranteed  provided that the conditions from Theorem \ref{thm:transverse_homoclinic} are met. 

Under the above assumptions, we have $\tsigma_\eps(\tz_\eps^-)=\tz_\eps^+$, and 
$\tsigma_0(\tz_0^-)=\tz_0^+$.  
We recall here that for the unperturbed system, the scattering map is the identity   $\tsigma_0=\textrm{Id}$, hence, in terms of action-angle coordinates $(I,\theta)$,  $I(\tz_0^-)=I(\tz_0^+)$, and $\theta(\tz_0^-)=\theta(\tz_0^+)$.

The result below describes the relation between $\tsigma_\eps$ and $\tsigma_0$ in terms of the action coordinate $I$.

\begin{thm}\label{prop:change_in_I}
The change in action $I$ by the scattering map $\tilde{\sigma}_\eps$ is given by:
\begin{equation}\label{eqn:change_in_I_non_ham}\begin{split}
  I\left(\tz^{+}_{\eps}\right)-I\left(\tz^-_{\eps}\right)
= &-\eps\int_{-\infty}^{+\infty}\left(\X^{1}I(\tPhi^{s}_{0}(\tz^\pm_{0}))
  -\X^{1}I(\tPhi^{s}_{0}(\tz_{0}))\right)ds\\
  &+O\left(\eps^{1+\rho}\right)\\
= &-\eps\int_{-\infty}^{+\infty}\left([I,H_1](\tPhi^{s}_{0}(\tz^\pm_{0}))
  -[I,H_1](\tPhi^{s}_{0}(\tz_{0}))\right)ds
  \\
  &+O\left(\eps^{1+\rho}\right).
\end{split}\end{equation}
where $\tz^+_0=\tz^-_0=\tz^\pm$, and $0<\varrho<1$.

The second formula  corresponds to the case when the perturbation is Hamiltonian. 
The integrals on the right-hand side are evaluated with $\X^1=\X^1(\cdot;0)$ and $H_1=H_1(\cdot;0)$, respectively.
\end{thm}
\begin{proof}
The key observation is that the unperturbed Hamiltonian $H_0$ for the rotator--pendulum system does not depend on $\theta$, hence $I$ is a slow variable, as it can be seen from \eqref{eqn:evolution3}.

The proof follows immediately from Lemma \ref{lem:I_of_Phi} below, by making $\sigma=0$, and subtracting
\eqref{eqn:I_of_Phi_minus} from  \eqref{eqn:I_of_Phi_plus}.
We also use the fact that for the unperturbed  pendulum-rotator system the foot points of the stable fiber and of the unstable fiber through the same point $\tz_0$ coincide, i.e., $\tz^-_0=\tz^+_0$ which we denote $\tz_0^\pm$.
\end{proof}

\begin{lem}\label{lem:I_of_Phi} For any $\varsigma\in\mathbb{R}$ we have
\begin{equation}\label{eqn:I_of_Phi_plus}
\begin{split}
  I&(\tPhi^{\varsigma}_{\eps}(\tz^{+}_{\eps}))
  -I(\tPhi^{\varsigma}_{\eps}(\tz_{\eps}))
  \\&=-\eps\int_{0}^{+\infty}\left(\X^{1}I(\tPhi^{s+\varsigma}_{0}(\tz^+_{0}))
  -\X^{1}I(\tPhi^{s+\varsigma}_{0}(\tz_{0}))\right)ds
  +O\left(\eps^{1+\rho}\right),\\
&=-\eps\int^{+\infty}_{0}\left([I,H_1](\tPhi^{s+\varsigma}_{0}(\tz^+_{0}))
  -[I,H_1](\tPhi^{s+\varsigma}_{0}(\tz_{0}))\right)ds
  +O\left(\eps^{1+\rho}\right),
\end{split}\end{equation}
and
\begin{equation}\label{eqn:I_of_Phi_minus}
\begin{split}
 I&(\tPhi^{\varsigma}_{\eps}(\tz^-_{\eps}))
  -I(\tPhi^{\varsigma}_{\eps}(\tz_{\eps}))
  \\&=+\eps\int_{-\infty}^{0}\left(\X^{1}I(\tPhi^{s+\varsigma}_{0}(\tz^-_{0}))
  -\X^{1}I(\tPhi^{s+\varsigma}_{0}(\tz_{0}))\right)ds
  +O(\eps^{1+\rho})
\\&=+\eps\int_{-\infty}^{0}\left([I,H_1](\tPhi^{s+\varsigma}_{0}(\tz^-_{0}))
  -[I,H_1](\tPhi^{s+\varsigma}_{0}(\tz_{0}))\right)ds
  +O(\eps^{1+\rho}).
  \end{split}\end{equation}
where $0<\rho<1$.

The second formula in each equation corresponds to the case when the perturbation is Hamiltonian. 
\end{lem}

\begin{proof}
We will only prove \eqref{eqn:I_of_Phi_plus}   as  \eqref{eqn:I_of_Phi_minus}   follows similarly.

We first apply Lemma \ref{lem:master_1} and Lemma \ref{lem:master_2} for $\bF=I$, obtaining
\begin{equation*}
\begin{split}
  I(\tPhi^{\varsigma}_{\eps}(\tz^+_{\eps}))
  -I(\tPhi^{\varsigma}_{\eps}(\tz_{\eps}))
  =-\int_{0}^{+\infty}&\left((\X^0 I+\eps \X^1 I)(\tPhi^{s+\varsigma}_{\eps}
(\tz^+_{\eps}))\right.\\
 &\left.-(\X^0 I+\eps \X^1 I)(\tPhi^{s+\varsigma}_{\eps}
(\tz_{\eps}))\right)ds.
\end{split}
\end{equation*}
Using \eqref{eqn:evolution3}, since $\X^0I=\frac{\partial{H_0}}{\partial\theta}=0$, we obtain
\begin{equation*}
\begin{split}
  I(\tPhi^{\varsigma}_{\eps}(\tz^+_{\eps}))
  -I(\tPhi^{\varsigma}_{\eps}(\tz_{\eps}))
  =-\eps\int_{0}^{+\infty}&\left(( \X^1 I)(\tPhi^{s+\varsigma}_{\eps}
(\tz^+_{\eps}))\right.\\
 &\left.-(\X^1 I)(\tPhi^{s+\varsigma}_{\eps}
(\tz_{\eps}))\right)ds.
\end{split}
\end{equation*}

Using Lemma \ref{lem:master_3}, we can replace the perturbed flow by the unperturbed flow by making  an error of order $O(\eps^\varrho)$, yielding
\begin{equation*}
\begin{split}
  I(\tPhi^{\varsigma}_{\eps}(\tz^+_{\eps}))
  -I(\tPhi^{\varsigma}_{\eps}(\tz_{\eps}))
  =-\eps\int_{0}^{+\infty}&\left(( \X^1 I)(\tPhi^{s+\varsigma}_{0}
(\tz^+_{0}))\right.\\
 &\left.-(\X^1 I)(\tPhi^{s+\varsigma}_{0}
(\tz_{0}))\right)ds\\ & +O(\eps^{1+\varrho}).
\end{split}
\end{equation*}

Finally, we note that in the pendulum-rotator system the foot-points of the stable fiber and of the unstable fiber through the same homoclinic point $\tz_0$ coincide, i.e., $\tz^-_0=\tz^+_0=\tz_0^\pm$.

In the case of the Hamiltonian perturbation,  we only need to substitute  $\X^1I=[I,H_1]$.
\end{proof}

\subsection{Change in angle by the scattering map}
{$ $}

Under the same assumptions as at the beginning of Section \ref{sec:pend_rot_change_action}, below we provide a result that describes the relation between $\tsigma_\eps$ and $\tsigma_0$ in terms of the angle coordinate $\theta$.

\label{sec:pend_rot_change_angle}
\begin{thm}\label{prop:change_in_theta}
The change in angle $\theta$ by the scattering map $\tilde{\sigma}_\eps$ is given by:
\begin{equation}
\begin{split}\label{eqn:delta_slow}
&\theta(\tz^{+}_{\eps})-\theta(\tz^{-}_{\eps})\\
=&-\eps\int_{-\infty}^{+\infty}\X^{1}\theta(\tPhi^{s}_{0}(\tz^{+}_{0}))
-\X^{1}\theta(\tPhi^{s}_{0}(\tz_{0})) ds\\
&+\eps\int_{0}^{+\infty} (\X^{1}I(\tPhi^{s}_{0}
(\tz^\pm_{0}))  -\X^{1}I(\tPhi^{s}_{0}(\tz_{0})))s ds \cdot\left(\frac{\partial^2 h_{0}}{\partial I^2}(I_0)\right) \\
 &+O(\eps^{1+\varrho})\\
=&-\eps\int_{-\infty}^{+\infty}[\theta,H_1](\tPhi^{s}_{0}(\tz^{+}_{0})
-[\theta,H_1](\tPhi^{s}_{0}(\tz_{0}) ds\\
&+\eps\int_{0}^{+\infty} ([I,H_1](\tPhi^{s}_{0}(\tz^\pm_{0}))  -[I,H_1](\tPhi^{s}_{0}(\tz_{0})))s ds \cdot\left(\frac{\partial^2 h_{0}}{\partial I^2}(I_0)\right)\\
 &+O(\eps^{1+\varrho}).
\end{split}
\end{equation}
where  $\tz^+=\tz^-=\tz^\pm$, $I_0=I(\tz^\pm)$, and $0<\varrho<1$. In the second term on the right-hand side the integral is thought of as a $1\times d$ vector,
and $\frac{\partial^2 h_{0}}{\partial I^2}(I_0)$ as a  $d\times d$ matrix. Also $[\theta,H_1]$, $[I,h_1]$ are $1\times d$ vector.

The second formula corresponds to the case when the perturbation is
Hamiltonian. The integrals on the right-hand side are evaluated with $\X^1=\X^1(\cdot;0)$ and $H_1=H_1(\cdot;0)$, respectively.

\end{thm}

\begin{proof}
Unlike in Theorem \ref{prop:change_in_I}, where $I$ is a slow variable, $\theta$ is a fast variable, as it can be seen from \eqref{eqn:evolution4}. However,  we will  show that the differences
$$\theta\left(\tz^{+}_{\eps}\right)-\theta\left(\tz_{\eps}\right)$$
and
$$\theta\left(\tz^{-}_{\eps}\right)-\theta\left(\tz_{\eps}\right)$$
are slow quantities. Then,  taking the difference,
$$\theta\left(\tz^{+}_{\eps}\right)-\theta\left(\tz^{-}_{\eps}\right)$$
is $O(\eps)$.

We begin with $\theta\left(\tz^{+}_{\eps}\right)-\theta\left(\tz_{\eps}\right)$.
Using Lemma \ref{lem:master_1} and Lemma \ref{lem:master_2} for $\bF=\theta$ we obtain
\begin{equation}\label{eqn:differece_theta_plus}
  \begin{split}
   \theta\left(\tz^{+}_{\eps}\right)-\theta\left(\tz_{\eps}\right)
   =
   -\int^{+\infty}_{0}&\left((\X^0\theta+\eps\X^1\theta)(\tPhi^{\varsigma}_{\eps}(\tz^{+}_{\eps}))\right.\\
   &\left.  -(\X^0\theta+\eps\X^1\theta)(\tPhi^{\varsigma}_{\eps}(\tz_{\eps}))\right)d\varsigma.
   \end{split}
\end{equation}

From \eqref{eqn:evolution4} we have
\[\X^0\theta+\eps\X^1\theta=\frac{\partial h_0}{\partial I}+\eps \X^1(\theta),\]
and \eqref{eqn:differece_theta_plus} becomes
\begin{equation}\label{eqn:two_integrals}\begin{split}
   -&\int^{+\infty}_{0}\left(\frac{\partial h_0}{\partial I}(\tPhi^{\varsigma}_{\eps}
   (\tz^{+}_{\eps}))-
   \frac{\partial h_0}{\partial I}(\tPhi^{\varsigma}_{\eps}(\tz_{\eps}))\right)d\varsigma\\
   -&\eps\int^{+\infty}_{0}\left((\X^1\theta)(\tPhi^{\varsigma}_{\eps}(\tz^{+}_{\eps}))-
   (\X^1\theta)(\tPhi^{\varsigma}_{\eps}(\tz_{\eps}))\right)d\varsigma.
   \end{split}
\end{equation}

The second integral in \eqref{eqn:two_integrals} has a factor of $\eps$, so we will focus on the first integral. Recall that $\frac{\partial h_{0}}{\partial I}$ depends only on $I$.
So the first integral in \eqref{eqn:two_integrals} can be written as
\[-\int^{+\infty}_{0}\left(\frac{\partial h_{0}}{\partial I}(I(\tPhi^{\varsigma}_{\eps}(\tz^{+}_{\eps})))
-\frac{\partial h_{0}}{\partial I}(I(\tPhi^{\varsigma}_{\eps}(\tz_{\eps})))\right)d\varsigma.\]

Let us first consider the case when $h_0$ is of one-degree-of-freedom, i.e.  $I\in \mathbb{R}$. 
We can use the integral  version of the Mean Value Theorem to rewrite
the integral. Recall,
\[f(x+\Delta x)-f(x)=\Delta x\int^{1}_{0}f'(x+t\Delta x)dt\]
Using
\begin{eqnarray*}
  f &=& \frac{\partial h_{0}}{\partial I} \\
  x+\Delta x &=& I\left(\tPhi^{\varsigma}_{\eps}\left(\tz^{+}_{\eps}\right)\right)  \\
  x &=& I\left(\tPhi^{\varsigma}_{\eps}\left(\tz_{\eps}\right)\right)
\end{eqnarray*}
the integral becomes
$$-\int^{+\infty}_{0}\left(\left({I}^{\varsigma,+}_{\eps}-{I}^{\varsigma}_{\eps}\right)
\int^{1}_{0}\frac{\partial^2 h_{0}}{\partial I^2}\left({I}^{\varsigma}_{\eps}+t\left( {I}^{\varsigma, +}_{\eps}-{I}^{\varsigma}_{\eps}\right)\right) dt\right)d\varsigma$$
where we denote ${I}^{\varsigma,+}_{\eps}=I\left(\tPhi^{\varsigma}_{\eps}\left(\tz^{+}_{\eps}\right)\right)$ and
${I}^{\varsigma}_{\eps}=I\left(\tPhi^{\varsigma}_{\eps}\left(\tz_{\eps}\right)\right)$.

We use Gronwall's inequality as in Lemma \ref{lem:Gronwall_application}
to rewrite the inside integral of the second partial derivative as
\[\int^{1}_{0}\frac{\partial^2 h_{0}}{\partial I^2}\left({I}^{\varsigma}_{\eps}+t({I}^{\varsigma,+}_{\eps}-{I}^{\varsigma}_{\eps})\right) dt
=\int^{1}_{0}\frac{\partial^2 h_{0}}{\partial I^2}\left({I}^{\varsigma}_{0}+t({I}^{\varsigma,+}_{0}-{I}^{\varsigma}_{0})\right) dt+O\left(\epsilon^\varrho\right).\]

Now $ {I}^{\varsigma,+}_{0}={I}^{\varsigma}_{0}=I_0$ because $I$ is constant along the unperturbed flow, hence the above integral equals \[\frac{\partial^2 h_{0}}{\partial I^2}\left( {I}_{0}\right)+O(\epsilon^\varrho).\]

We now apply   Lemma \ref{lem:I_of_Phi} to rewrite $ {I}^{\varsigma,+}_{\eps}-\tilde{I}^{\varsigma}_{\eps}$, so the integral becomes
\begin{equation}\label{eqn:first_int}
\begin{split}
    &\eps\int^{+\infty}_{0}\int^{+\infty}_{0}(\X^{1}I(\tPhi^{s+\varsigma}_{0}
(\tz^+_{0}))
  -\X^{1}I(\tPhi^{s+\varsigma}_{0}(\tz_{0})))  d\varsigma ds\cdot \left(\frac{\partial^2 h_{0}}{\partial I^2}
  (I_0)\right)\\
 &\qquad +O(\eps^{1+\varrho})
\end{split}
\end{equation}
This integral has a factor of $\eps$, and the remaining term is $O(\eps^{1+\varrho})$, thus $\theta\left(\tz^{+}_{\eps}\right)-\theta\left(\tz_{\eps}\right)$ is a slow quantity.

Denote by $\mathscr{I}$ the antiderivative of
\[ s\mapsto(\X^{1}I(\tPhi^{s}_{0}
(\tz^\pm_{0}))  -\X^{1}I(\tPhi^{s}_{0}(\tz_{0})))\]
which approaches $0$ as $s\to\pm\infty$; we recall here that $\tz_0^+=\tz^-_0=\tz^\pm_0$.
We have
\begin{equation}\label{eqn:antiderivative}
\begin{split}\mathscr{I}(s)=&-\int_{s}^{+\infty}(\X^{1}I(\tPhi^{\upsilon}_{0}
(\tz^+_{0}))  -\X^{1}I(\tPhi^{\upsilon}_{0}(\tz_{0})))  d\upsilon\\
&=\int_{-\infty}^{s}(\X^{1}I(\tPhi^{\upsilon}_{0}
(\tz^-_{0}))  -\X^{1}I(\tPhi^{\upsilon}_{0}(\tz_{0})))d\upsilon.
\end{split}\end{equation}

Making the change of variable $\upsilon =s+\varsigma$  the integral in \eqref{eqn:first_int} becomes
\begin{equation}\label{eqn:first_int_2}
\begin{split}
    \int_{0}^{+\infty}\int_{s}^{+\infty}(\X^{1}I(\tPhi^{\upsilon}_{0}
(\tz^+_{0}))
  -\X^{1}I(\tPhi^{\upsilon}_{0}(\tz_{0})))  d\upsilon ds=
 -\int_{0}^{+\infty}\mathscr{I}(s)ds.
\end{split}
\end{equation}

Using Integration by Parts we obtain
\begin{equation}\begin{split}
-\int_{0}^{+\infty}\mathscr{I}(s)ds=&-s\mathscr{I}(s)\biggr| _{0} ^{+\infty}
+\int_{0}^{+\infty} (\X^{1}I(\tPhi^{s}_{0} (\tz^\pm_{0}))  -\X^{1}I(\tPhi^{s}_{0}(\tz_{0})))s ds\\
=&\int_{0}^{+\infty} (\X^{1}I(\tPhi^{s}_{0}
(\tz^\pm_{0}))  -\X^{1}I(\tPhi^{s}_{0}(\tz_{0})))s ds.
\end{split}\end{equation}
In the above, the quantity $s\mathscr{I}(s)$ obviously equals to $0$  at $s=0$, and
equals to $0$ when $s\to\infty$ since, by l'Hopital Rule
\[\lim_{s\to\infty}\frac{\mathscr{I}(s)}{s^{-1}}=\lim_{s\to\infty}-\frac{(\X^{1}I(\tPhi^{s}_{0}
(\tz^\pm_{0}))  -\X^{1}I(\tPhi^{s}_{0}(\tz_{0})))}{s^2}=0,\]
since $(\X^{1}I(\tPhi^{s}_{0}
(\tz^\pm_{0}))  -\X^{1}I(\tPhi^{s}_{0}(\tz_{0})))$ approaches $0$ at exponential rate.

Applying Lemma \ref{lem:master_3} to the second integral in \eqref{eqn:two_integrals}, and combining with the above we obtain
\begin{equation}\label{eqn:first_int_plus}
\begin{split}
&\theta\left(\tz^{+}_{\eps}\right)-\theta\left(\tz_{\eps}\right)\\
&\quad=+\eps \int_{0}^{+\infty} (\X^{1}I(\tPhi^{s}_{0}
(\tz^\pm_{0}))  -\X^{1}I(\tPhi^{s}_{0}(\tz_{0})))s ds \cdot\left(\frac{\partial^2 h_{0}}{\partial I^2}(I_0)\right)\\ &\qquad-\eps\int^{+\infty}_{0}\X^{1}\theta\left(\tPhi^{s}_{0}\left(\tz^{+}_{0}\right)\right)
-\X^{1}\theta\left(\tPhi^{s}_{0}\left(\tz_{0}\right)\right) ds\\
 &\qquad +O(\eps^{1+\varrho}).
\end{split}
\end{equation}

Similarly, for $\theta\left(\tz^{-}_{\eps}\right)-\theta\left(\tz_{\eps}\right)$
we obtain an expression as a sum of two integrals
\begin{equation}\label{eqn:first_int_minus}
\begin{split}
&\theta\left(\tz^{-}_{\eps}\right)-\theta\left(\tz_{\eps}\right)\\
&\quad=+\eps\int_{-\infty}^{0}\X^{1}\theta(\tPhi^{s}_{0}\left(\tz^{+}_{0}\right))
-\X^{1}\theta(\tPhi^{s}_{0}\left(\tz_{0}\right)) ds\\
&\qquad-\eps \int_{-\infty}^{0}
(\X^{1}I(\tPhi^{s}_{0} (\tz^\pm_{0}))  -\X^{1}I(\tPhi^{s}_{0}(\tz_{0})))s ds
\cdot \left(\frac{\partial^2 h_{0}}{\partial I^2}(I_0)\right)\\
 &\qquad +O(\eps^{1+\varrho}).
\end{split}
\end{equation}

In the case when $d=1$, recalling that $\tz_0^+=\tz^-_0=\tz^\pm_0$, we conclude
that
\begin{equation*}
\begin{split}
\theta(\tz^{+}_{\eps})&-\theta(\tz^{-}_{\eps})\\
=&-\eps\int_{-\infty}^{+\infty}\X^{1}\theta(\tPhi^{s}_{0}\left(\tz^{+}_{0}\right))
-\X^{1}\theta(\tPhi^{s}_{0}\left(\tz_{0}\right)) ds\\
&+\eps\int_{0}^{+\infty} (\X^{1}I(\tPhi^{s}_{0}
(\tz^\pm_{0}))  -\X^{1}I(\tPhi^{s}_{0}(\tz_{0})))s ds\cdot \left(\frac{\partial^2 h_{0}}{\partial I^2}(I_0)\right) \\
 &+O(\eps^{1+\varrho})
\end{split}
\end{equation*}

In the case where $I \in \mathbb{R}^d$, we can use the vectorial version of the Mean Value Theorem. 
For $f:\mathbb{R}^d\rightarrow \mathbb{R}$, we have
\[f(\mathbf{x}+t\Delta\mathbf{x})-f(\mathbf{x})=\left\langle\Delta\mathbf{x}, \int^{1}_{0}\nabla f(\mathbf{x}+t\Delta\mathbf{x}) dt\right\rangle,\]
where $\langle \cdot,\cdot\rangle$ denotes the inner product on $\mathbb{R}^d$.

Setting
\begin{eqnarray*}
  f &=& \frac{\partial h_{0}}{\partial I_j} \\
  \mathbf{x+\Delta\mathbf{x}} &=& {I}^{\varsigma, +}_{\eps}  \\
  \mathbf{x} &=& {I}^{\varsigma}_{\eps}
\end{eqnarray*}
and proceeding as before, the first integral that appears  in the computation of $\theta_j\left(\tz^{+}_{\eps}\right)-\theta_j\left(\tz_{\eps}\right)$
becomes
\begin{equation*}
\begin{split}&-\int_{0}^{+\infty}\left(\frac{\partial h_{0}}{\partial I_j}(I(\tPhi^{\varsigma}_{\eps}(\tz^{+}_{\eps})))
-\frac{\partial h_{0}}{\partial I_j}(I(\tPhi^{\varsigma}_{\eps}(\tz_{\eps})))\right)d\varsigma
\\
=&   +\eps\int^{+\infty}_{0}\int^{+\infty}_{0}\left\langle  \X^{1}I (\tPhi^{s+\varsigma}_{0}(\tz_{0} ) )
  -\X^{1}I(\tPhi^{s+\varsigma}_{0}(\tz_{0}))  , \frac{\partial^2 h_{0}}{\partial I\partial I_j}(I_0)))\right\rangle d\varsigma  ds\\
&+O (\eps^{1+\varrho} ),\\
=&-\eps\int_{0}^{+\infty} \left\langle  \mathscr{I}(s) , \frac{\partial^2 h_{0}}{\partial I\partial I_j}(I_0)\right\rangle    ds\\
&+O (\eps^{1+\varrho} ),
\end{split}
\end{equation*}
where we now denote by $\mathscr{I}(s)$  the vector-valued function whose component
$\mathscr{I}_i(s)$ represents the antiderivative of \[ s\mapsto(\X^{1}I_i(\tPhi^{s}_{0}
(\tz^\pm_{0}))  -\X^{1}I_i(\tPhi^{s}_{0}(\tz_{0})))\]
which approaches $0$ as $s\to\pm\infty$, for $i=1,\ldots,d$.

Using Integration by Parts the last expression can be written as
\begin{equation*}\begin{split}
&+\eps\int_{0}^{+\infty}  \left\langle  \X^1I(\Phi_0^s(\tz_0^\pm)-\X^1I(\Phi_0^s(\tz_0),
\frac{\partial^2 h_{0}}{\partial I\partial I_j}(I_0)\right\rangle   s ds\\
&+O (\eps^{1+\varrho} ).
\end{split}
\end{equation*}

The second integral that appears in the computation of $\theta_j\left(\tz^{+}_{\eps}\right)-\theta_j\left(\tz_{\eps}\right)$
has the same form as in the $1$-dimensional case $d=1$.

Thus, for the vector $\theta\left(\tz^{+}_{\eps}\right)-\theta\left(\tz_{\eps}\right)$ we obtain
\begin{equation}\label{eqn:first_int_minus}
\begin{split}
&\theta\left(\tz^{+}_{\eps}\right)-\theta\left(\tz_{\eps}\right)\\
&\quad=+\eps \int_{0}^{+\infty} (\X^{1}I(\tPhi^{s}_{0}
(\tz^\pm_{0}))  -\X^{1}I(\tPhi^{s}_{0}(\tz_{0})))s ds \cdot \left(\frac{\partial^2 h_{0}}{\partial I^2}(I_0)\right)\\ &\qquad-\eps\int^{+\infty}_{0}\left(\X^{1}\theta(\tPhi^{s}_{0}(\tz^{+}_{0}))
-\X^{1}\theta(\tPhi^{s}_{0}(\tz_{0})) \right)ds\\
 &\qquad +O(\eps^{1+\varrho}),
\end{split}
\end{equation}
where in the first expression on the right-hand side the integral is thought of as a $1\times d$ vector,
and $\frac{\partial^2 h_{0}}{\partial I^2}(I_0)$ as a  $d\times d$ matrix.

Computing $\theta_j\left(\tz^{-}_{\eps}\right)-\theta_j\left(\tz_{\eps}\right)$ in a similar fashion and combining with the above we conclude
\begin{equation*}
\begin{split}
\theta_j(\tz^{+}_{\eps})&-\theta_j(\tz^{-}_{\eps})\\
&\quad=+\eps \int_{-\infty}^{+\infty} (\X^{1}I(\tPhi^{s}_{0}
(\tz^\pm_{0}))  -\X^{1}I(\tPhi^{s}_{0}(\tz_{0})))s ds \cdot \left(\frac{\partial^2 h_{0}}{\partial I^2}(I_0)\right)\\ &\qquad-\eps\int_{-\infty}^{+\infty}\left(\X^{1}\theta(\tPhi^{s}_{0}(\tz^{+}_{0}))
-\X^{1}\theta(\tPhi^{s}_{0}(\tz_{0})) \right)ds\\
 &\qquad +O(\eps^{1+\varrho}).
\end{split}
\end{equation*}
\end{proof}

\subsection{Comparison with similar results}\label{sec:comparison_similar}
Consider the special case when the perturbation $\X^1$ is Hamiltonian and time-periodic in $t$, i.e.,
$\X^1=J\nabla H_1$ for some $H_1=H_1(z,t)$, with $t\in\mathbb{T}^1=\mathbb{R}/\mathbb{Z}$.
Then the scattering map  is exact symplectic map and depends smoothly on parameters, in particular on $\eps$, so it can be computed perturbatively.  More precisely, the  scattering map,   in terms of a local system of coordinates $(I,\theta,t)$ on $\tilde\Lambda_\eps$, can be expanded as a power of $\eps$ as follows:
\begin{equation}\label{eqn:scattering_perurbative_1}
   \tilde\sigma_\eps=\tilde\sigma_0+\eps J\nabla S\circ \tilde\sigma_0+O(\eps^2),
\end{equation}
where $\mathcal{S}_0$ is a $C^\ell$-smooth Hamiltonian function defined on some open subset of $\tilde\Lambda_\eps$.
Hence  $J\nabla\mathcal{S}_0$ represents a Hamiltonian vector field on $\tilde\Lambda_\eps$.
This formula is no longer true in the case of perturbations that are not Hamiltonian.
See \cite{DelshamsLS08a}.

In the case of the pendulum-rotator system,
since $\tilde\sigma_0 =\textrm{Id}$, we have
\begin{equation}\label{eqn:scattering_perurbative_1}
   \tilde\sigma_\eps=\textrm{Id}+\eps J\nabla S +O(\eps^{1+\varrho}),
\end{equation}
and the Hamiltonian function $S$ that generates the scattering map can be computed explicitly as follows.
Let
\begin{equation}
\label{separatrix}
\begin{split}(p^0(\tau+t\bar 1), q^0 (\tau+t\bar 1))=(p^0_1(\tau_1+t),\ldots,p^0_n(\tau_n+t), q^0_1(\tau_1+t),\ldots,q^0_n(\tau_n+t)),
\end{split}
\end{equation}
be a parametrization of the system of separatrices of the penduli, where
$\tau=(\tau_1,\ldots,\tau_n)\in\mathbb{R}^n$  and  $\bar 1=(1,\ldots,1)\in\mathbb{R}^n$.

\begin{equation}\label{eqn:L}
\begin{split}
   L(\tau,I,\theta,t)=-\int_{-\infty}^{+\infty} &\left(H_1(p^0(\tau+t\bar 1), q^0 (\tau+t\bar 1), I, \theta+\omega(I)s,t+s)\right.\\
   &\left. -H_1(0, 0, I, \theta+\omega(I)s,t+s)\right)ds
\end{split}
\end{equation}

Assume that  the map
\begin{equation}\label{eqn:tau_star}
\begin{split}
   \tau\in\R^n\mapsto L(\tau,I,\theta,t)\in\R
\end{split}
\end{equation}
has a has a non-degenerate critical point $\tau^*$, which is locally given, by the implicit
function theorem, by
\begin{equation}\label{eqn:tau_star}
\begin{split}
   \tau^*=\tau^*(I,\theta,t).
\end{split}
\end{equation}
Hence
\begin{equation}\label{eqn:L_tau_star}
\frac{\partial L}{\partial \tau}(\tau^*(I,\theta,t), I,\theta,t)=0.
\end{equation}

Then define the auxiliary function $\mathcal{L}$ by
\begin{equation}\label{eqn:mathcalL}
\begin{split}
    \mathcal{L}(I,\theta,t)=L(\tau^*(I,\theta,t), I,\theta,t).
\end{split}
\end{equation}

It is not difficult to show that $\mathcal{L}$ satisfies the following relation for all $\sigma\in\R$:
\begin{equation}\label{eqn:mathcalaL_identity}
\begin{split}
    \mathcal{L}(I,\theta,t)=\mathcal{L}(I,\theta-\omega(I)\sigma, t-\sigma).
\end{split}
\end{equation}
In particular, for $\sigma=t$, we have $\mathcal{L}(I,\theta,t)=\mathcal{L}(I,\theta-\omega(I)t,0)$. If we denote by $\mathcal{L}^*$ the function defined by
\begin{equation}\label{eqn:mathcalaLstar}
\begin{split}
    \mathcal{L}^*(I,\bar\theta)=\mathcal{L}(I,\theta-\omega(I)t, 0),\textrm { for } \bar\theta=\theta-\omega(I)t,
\end{split}
\end{equation}
then
\begin{equation}\label{eqn:mathcalaLstar}
\begin{split}
    \mathcal{L}(I,\theta,t)=\mathcal{L}^*(I,\bar\theta), \textrm { for } \bar\theta=\theta-\omega(I)t.
\end{split}
\end{equation}
This says that the function $\mathcal{L}= \mathcal{L}(I,\theta,t)$, while nominally a function of three variables,
it depends in fact on two variables only.

It turns out that the Hamiltonian function $S$ that generates the scattering map is given by
\begin{equation}\label{eqn:S} S(I,\theta,t)=-L(I,\theta,t).
\end{equation}

For $\tilde\sigma_\eps(I^-,\theta^-,t^-)=(I^+,\theta^+,t^+)$, from \eqref{eqn:scattering_perurbative_1} we obtain
\begin{eqnarray}
\label{eqn:comp_1} I^{+}-I^{-}=&\displaystyle\eps\frac{\partial \mathcal{L}}{\partial \theta}(I,\theta,t)+O(\eps^{1+\varrho}),\\
\label{eqn:comp_2}  \theta^{+}-\theta^{-}=&\displaystyle-\eps\frac{\partial \mathcal{L}}{\partial I}(I,\theta,t)+O(\eps^{1+\varrho}),\\
\label{eqn:comp_3}  t^+-t^-=&\displaystyle 0.
\end{eqnarray}

From \eqref{eqn:mathcalL} and \eqref{eqn:L_tau_star}
\begin{equation}
\label{eqn:partial_mathcal_L}\begin{split}
\frac{\partial \mathcal{L}}{\partial I}(I,\theta,t)=&\frac{\partial L}{\partial \tau}(\tau^*(I,\theta,t),I,\theta,t)\frac{\partial \tau^*}{\partial I}(I,\theta,t)+
\frac{\partial L}{\partial I}(\tau^*(I,\theta,t),I,\theta,t)\\
=&\frac{\partial L}{\partial I}(\tau^*(I,\theta,t),I,\theta,t),\\
\frac{\partial \mathcal{L}}{\partial \theta}(I,\theta,t)=&\frac{\partial L}{\partial \tau}(\tau^*(I,\theta,t),I,\theta,t)\frac{\partial \tau^*}{\partial \theta}(I,\theta,t)+
\frac{\partial L}{\partial I}(\tau^*(I,\theta,t),I,\theta,t)\\
=&\frac{\partial L}{\partial \theta}(\tau^*(I,\theta,t),I,\theta,t).
\end{split}\end{equation}

From \eqref{eqn:L}, and using the fact that $\X^1I=[I,H_1]=-\frac{\partial H_1}{\partial \theta}$,  we obtain:
\begin{equation}
\label{eqn:partial_mathcal_L_theta}\begin{split}
\frac{\partial \mathcal{L}}{\partial \theta}(I,\theta,t)=&-\int_{-\infty}^{+\infty} \left(\frac{\partial H_1}{\partial \theta} (p^0(\tau+t\bar 1), q^0 (\tau+t\bar 1), I, \theta+\omega(I)s,t+s)\right.\\
   &\left.\qquad\quad\, -\frac{\partial H_1}{\partial \theta} (0, 0, I, \theta+\omega(I)s,t+s)\right)ds\\
   =&-\int_{-\infty}^{+\infty} \left([I,H_1] (0, 0, I, \theta+\omega(I)s,t+s)\right.\\
   &\left.\qquad\quad\, -[I,H_1] (p^0(\tau+t\bar 1), q^0 (\tau+t\bar 1), I, \theta+\omega(I)s,t+s)\right)ds.
\end{split}\end{equation}

Above, note that the point $(0, 0, I, \theta+\omega(I)s,t+s)$ corresponds to $\tz_0^\pm$, and the point $(p^0(\tau+t\bar 1), q^0 (\tau+t\bar 1), I, \theta+\omega(I)s,t+s)$ corresponds to $\tz_0$ in Section \ref{sec:pend_rot_change_action}.
Thus, the formula for the change in the action by the scattering map in \eqref{eqn:comp_1} is the same as the one
given in Theorem \ref{prop:change_in_I}.

From \eqref{eqn:L}, and using that
$\X^1\theta=[\theta,H_1]=\frac{\partial H_1}{\partial I}$,
$\X^1I=[I,H_1]=-\frac{\partial H_1}{\partial \theta}$,
we obtain:
\begin{equation}
\label{eqn:partial_mathcal_L_I}\begin{split}
&\frac{\partial \mathcal{L}}{\partial I}(I,\theta,t)\\
&=-\int_{-\infty}^{+\infty} \left(\frac{\partial H_1}{\partial I} (p^0(\tau+t\bar 1), q^0 (\tau+t\bar 1), I, \theta+\omega(I)s,t+s)\right.\\
   &\qquad\qquad\, \left.-\frac{\partial H_1}{\partial I} (0, 0, I, \theta+\omega(I)s,t+s)\right)ds\\
   &\quad-\int_{-\infty}^{+\infty} \left(\frac{\partial H_1}{\partial \theta} (p^0(\tau+t\bar 1), q^0 (\tau+t\bar 1), I, \theta+\omega(I)s,t+s)\right.\\
   &\qquad\qquad\,\left. -[I,H_1] (0, 0, I, \theta+\omega(I)s,t+s)\right)(D_I\omega (I)s)ds\\
   =&\int_{-\infty}^{+\infty} \left([\theta, H_1] (0, 0, I, \theta+\omega(I)s,t+s)\right.\\
   &\qquad\qquad\, \left.-[\theta, H_1](p^0(\tau+t\bar 1), q^0 (\tau+t\bar 1), I, \theta+\omega(I)s,t+s)\right)ds\\
   &\quad-\int_{-\infty}^{+\infty} \left(I,H_1](0, 0, I, \theta+\omega(I)s,t+s) \right.\\
   &\quad\qquad\,\left. -[I,H_1](p^0(\tau+t\bar 1), q^0 (\tau+t\bar 1), I, \theta+\omega(I)s,t+s) \right)(D_I\omega (I)s)ds.
\end{split}\end{equation}

Since  $D\omega(I)=\frac{\partial^2{h_0}}{\partial I^2}(I)$, and noting that this it is independent of the variable of integration, so it can be moved outside of the integral,  the formula for the change in the angle by the scattering map in \eqref{eqn:comp_2} is the same as the one
given in Theorem \ref{prop:change_in_theta}.

\appendix
\section{Gronwall's inequality}\label{sec:gronwall}
In this section we apply  Gronwall's Inequality  to estimate the error between the solution of an unperturbed system
and the solution of the perturbed system, over a time of logarithmic order with respect to the size of the perturbation.
\begin{thm}[Gronwall's Inequality]
Given a  continuous  real valued function  $\phi\geq 0$, and   constants $\delta_0,\delta_1\geq 0$, $\delta_2>0$,
if
\begin{equation}
\phi(t)\leq \delta_0+\delta_1(t-t_0)+\delta_2\int_{t_0}^{t}\phi(s)ds
\end{equation}
then
\begin{equation}
\phi(t)\leq \left (\delta_0+\frac{\delta_1}{\delta_2} \right)e^{\delta_2(t-t_0)}-\frac{\delta_1}{\delta_2}.
\end{equation}
\end{thm}
For a reference, see, e.g., \cite{verhulst2006nonlinear}.

\begin{lem}\label{lem:Gronwall_application}
Consider the following differential equations:
\begin{eqnarray}
\label{eqn:eq_0}\dot{z}(t)&=&\X^0(z,t)\\
\label{eqn:eq_1}\dot{z}(t)&=&\X^0(z,t)+\eps \X^1(z,t,\eps)
\end{eqnarray}
Assume that $\X^0,\X^1$ are    uniformly Lipschitz continuous  in the variable $z$,  $C_0$ is the Lipschitz constant of $\X^0$,  and $\X^1$ is bounded with $\|\X^1\|\leq C_1$, for some $C_0,C_1>0$.
Let $z_0$ be a solution of the equation \eqref{eqn:eq_0} and $z_\eps$ be a solution of the equation \eqref{eqn:eq_1} such that
\begin{equation}\label{eqn:eq_3}
\|z_0(t_0)-z_\eps(t_0)\|<c\eps.
\end{equation}
Then, for  $0<\varrho_0<1$, $k\leq \frac{1-{\varrho_0}}{C_0}$, and $K=c+\frac{C_1}{C_0}$, we have
\begin{equation}\label{eqn:eq_4}
\|z_0(t)-z_\eps(t)\|< K\eps^{\varrho_0}, \textrm{ for } 0\leq t-t_0\leq k\ln(1/\eps).
\end{equation}
\end{lem}

\begin{proof}
For $z_0$ and $z_\eps$ solutions of \eqref{eqn:eq_0} and \eqref{eqn:eq_1}, respectively, we have
\begin{eqnarray}
\label{eqn:eq_5}  z_0(t)&=&z_0(t_0)+\int_{t_0}^{t}\X^0(z_0(s),s)ds,\\
\label{eqn:eq_6}  z_\eps(t)&=&z_\eps(t_0)+\int_{t_0}^{t}\X^0(z_\eps(s),s)ds+\eps\int_{t_0}^{t_1}\X^1(z_\eps(s),s)ds.
\end{eqnarray}
Subtracting,  we obtain
\begin{equation}\label{eqn:eq_7}\begin{split}
\|z_\eps(t)-z_0(t)\|\leq \|z_\eps(t_0)-z_0(t_0)\|&+\int_{t_0}^{t}\|\X^0(z_\eps(s),s)-\X^0(z_0(s),s)\|ds\\&
+\eps\int_{t_0}^{t}\|\X^1(z_\eps(s),s)\|ds.
\end{split}\end{equation}
Using \eqref{eqn:eq_3} for the first term on the right-hand  side, the Lipschitz condition on $X^0$ for the second, and the boundedness of $\X^1$ for the third we obtain:
\begin{equation}\label{eqn:eq_8}\begin{split}
\|z_\eps(t)-z_0(t)\|\leq c\eps&+C_0\int_{t_0}^{t}\|z_\eps(s)-z_0(s)\|ds\\&
+\eps C_1(t-t_0).
\end{split}\end{equation}
Applying the Gronwall inequality for $\delta_0=c$, $\delta_1=\eps  C_1$, and $\delta_2=C_0$, and recalling that $K=c+\frac{C_1}{C_0}$ we obtain
\begin{equation}\label{eqn:eq_9}\begin{split}
\|z_\eps(t)-z_0(t)\|&\leq \eps \left (c+\frac{C_1}{C_0}\right)e^{C_0(t-t_0)}-\eps\frac{C_1}{C_0}\\
&\leq \eps  K e^{C_0(t-t_0)}.
\end{split}\end{equation}
If we let $0\leq t-t_0\leq k\ln(1/\eps)$ we obtain
\begin{equation}\label{eqn:eq_10}\begin{split}
\|z_\eps(t)-z_0(t)\|&\leq \eps \left (c+\frac{C_1}{C_0}\right) e^{C_0(t-t_0)}-\eps\frac{C_1}{C_0}\\
&\leq \eps  K e^{C_0k\ln(1/\eps)}\\&=\eps K \left(\frac{1}{\eps}\right)^{C_0 k}.
\end{split}\end{equation}
Since  $k\leq \frac{1-\varrho}{C_0}$ we conclude
\begin{equation}\label{eqn:eq_11}\begin{split}
\|z_\eps(t)-z_0(t)\|&\leq\eps K \left(\frac{1}{\eps}\right)^{1-\varrho}=K\eps^{\varrho}.
\end{split}\end{equation}
\end{proof}

We note that, with the above argument, for a time of logarithmic order with respect to the size of the perturbation,
we can only obtain an error of order $O(\eps^{\varrho})$ with $0<\rho <1$, but we cannot obtain an error of order $O(\eps)$.

\bibliographystyle{alpha}
\bibliography{diffusion}

\end{document}